\documentclass[11pt]{article}
\title{\textsc{Largest integral simplices with one interior integral point: Solution
of Hensley's conjecture and related results}}
\author{Gennadiy Averkov\footnote{Institute of Mathematical Optimization, Faculty of Mathematics, University of Magdeburg, Universit\"atsplatz 2, 39106 Magdeburg, Germany.
Emails: averkov@ovgu.de, jan.kruempelmann@ovgu.de} \qquad Jan Kr\"umpelmann$^*$ \qquad Benjamin Nill\footnote{Department of Mathematics, Stockholm University, 10691 Stockholm, Sweden. Email: nill@math.su.se}}
\usepackage{amsthm,amsmath,amssymb,enumerate,graphicx,graphpap,empheq}
\usepackage{srcltx}
\usepackage{paralist}
\usepackage{color}
\usepackage{soul}
\usepackage{bbm} 

\setuldepth{a}

\newcommand{\rcg}{\operatorname{rcg}}
\newcommand{\cF}{\mathcal{F}}
\newcommand{\xquer}{y}

\newcommand{\ORD}{\operatorname{ORD}}
\newcommand{\PS}{\operatorname{PS}}
\newcommand{\rmcmd}[1]{\mathop{\mathrm{#1}}\nolimits}

\newcommand{\R}{\mathbb{R}}
\newcommand{\N}{\mathbb{N}}
\newcommand{\Z}{\mathbb{Z}}

\newcommand{\setcond}[2]{\left\{#1 \, : \, #2 \right\}}

\newcommand{\sprod}[2]{\left< #1 \, , \, #2 \right>}

\newcommand{\vol}{\rmcmd{vol}}

\newcommand{\modulo}[1]{\, (\operatorname{mod}\, #1)}
\newcommand{\ld}[1]{\, \operatorname{ld}(#1)\,}

\newcommand{\vertset}[1]{\, \operatorname{vert}(#1)}
\newcommand{\conv}{\operatorname{conv}}
\newcommand{\intr}[1]{\operatorname{int}(#1)}
\newcommand{\relintr}[1]{\operatorname{relint}(#1)}
\newcommand{\bd}[1]{\operatorname{bd}(#1)}
\newcommand{\relbd}[1]{\operatorname{relbd}(#1)}
\newcommand{\Chi}{\mathcal{X}}
\newcommand{\Y}{\mathcal{Y}^n}

\newcommand{\IK}{\operatorname{IK}}
\newcommand{\len}{\operatorname{len}}
\newcommand{\ca}{\operatorname{ca}}
\newcommand{\Q}{\mathbb{Q}}
\newcommand{\allone}{\mathbbm{1}}
\renewcommand{\P}{\mathbb{P}}
\newcommand{\lspan}{\operatorname{lin}}
\newcommand{\thmtitle}[1]{\textup{(#1.)}}
\newcommand{\Pd}[1]{\mathcal{P}^d({#1})}
\newcommand{\Sd}[1]{\mathcal{S}^d({#1})}
\newcommand{\Plm}{\mathcal{P}^d_{0, \operatorname{lmax}}}
\newcommand{\Pim}{\mathcal{P}^d_{0, \operatorname{imax}}}

\newtheorem{nn}{}[section]
\newtheorem{theorem}[nn]{Theorem}

\newtheorem{proposition}[nn]{Proposition}
\newtheorem{lemma}[nn]{Lemma}
\newtheorem{corollary}[nn]{Corollary}

\newcommand{\rescite}[2]{\textup{\cite[{#2}]{#1}}}
\newtheorem{claim}{Claim}[nn]

\theoremstyle{definition}
\newtheorem*{acknowledgments*}{Acknowledgments}
\newtheorem{remark}[nn]{Remark}

\numberwithin{equation}{section}

\begin{document}

\maketitle

\begin{abstract}
For each dimension $d$, $d$-dimensional integral
simplices with exactly one interior integral point have bounded
volume. This was first shown by Hensley. Explicit volume bounds were
determined by Hensley, Lagarias and Ziegler, Pikhurko, and Averkov. In this
paper we determine the exact upper volume bound for such simplices
and characterize the volume-maximizing simplices. We also
determine the sharp upper bound on the coefficient of asymmetry of an integral polytope with a single interior
integral point. This result confirms a conjecture of Hensley from 1983. Moreover, for an integral simplex with precisely one interior
integral point, we give bounds on the volumes of its faces, the barycentric coordinates of the interior integral point and its number of integral points. 
Furthermore, we prove a bound on the
lattice diameter of integral polytopes with a fixed number of interior integral points.
The presented results have applications in toric geometry and in integer optimization.
\end{abstract}

\newtheoremstyle{itsemicolon}{}{}{\mdseries\rmfamily}{}{\itshape}{:}{ }{}
\newtheoremstyle{itdot}{}{}{\mdseries\rmfamily}{}{\itshape}{:}{ }{}
\theoremstyle{itdot}
\newtheorem*{msc*}{2010 Mathematics Subject Classification} 

\begin{msc*}
	Primary: 52B20, Secondary: 14M25, 90C11.
\end{msc*}

\newtheorem*{keywords*}{Keywords}

\begin{keywords*}
	barycentric coordinates; simplex; toric Fano variety; volume; integral polytope; lattice-free polytope; lattice diameter; Sylvester sequence
\end{keywords*}

\section{Introduction}\label{sec:introduction}

\subsection*{Background information}

The main objective of this manuscript is to provide sharp upper bounds on the size of integral polytopes with exactly one interior integral point.
Note that standard notation and terminology is defined at the beginning of Section~\ref{sec:results}.
For given $k \in \N \cup \{0\}$, we introduce the following two families of polytopes: 
\begin{align*}
 & \Pd{k} := \setcond{P \subseteq \R^d}{P \; \text{polytope}, \; \dim(P) = d, \; \vertset{P} \subseteq \Z^d, \; |\intr{P} \cap \Z^d| = k}\\
\text{and} \\
 & \Sd{k} := \setcond{S \in \Pd{k}}{S \; \text{is a simplex}}.
\end{align*}
Given an integer $k \ge 1$ and a dimension $d \in \N$, the set $\Pd{k}$ is finite up to affine transformations
which preserve the integer lattice $\Z^d$. In particular, $\Sd{k}$ is finite as well.
A natural way to prove finiteness of $\Pd{k}$ and $\Sd{k}$ is to bound the volume of their elements from above in terms of $k$ and $d$. 
This was first done by Hensley \cite{MR688412}.

In order to obtain a volume bound for $\Sd{1}$, Hensley proved a lower bound on
the minimal barycentric coordinate of the single interior integral point of $S \in \Sd{1}$. The minimal barycentric coordinate is in one-to-one correspondence
to the so-called coefficient of asymmetry of $S$ about its interior integral point. The coefficient of asymmetry is defined as follows: the intersection of each line through
the interior integral point with $S$ is divided into two parts by this point. For each line, consider the ratio between the lengths of these two parts.
Then the coefficient of asymmetry is the maximum of these ratios. It is easy to observe that for $S \in \Sd{1}$,
the coefficient of asymmetry about its interior integral point is equal to $\frac{1}{\beta} - 1$,
where $\beta$ is the minimal barycentric coordinate of this point; see also \rescite{MR1996360}{(3)}.
Hensley's results led him to conjecture the following:
The simplex $\conv(\{o, s_1 e_1, \ldots, s_d e_d\}) \in \Sd{1}$ has maximal coefficient of asymmetry among all elements of $\Sd{1}$, where
$(s_i)_{i \in \N}$ denotes the \emph{Sylvester sequence}, which is given by 
\begin{align*}
 & s_1 := 2, \\
 & s_i := 1 + \prod_{j=1}^{i-1} s_j \qquad \text{for} \;  i \ge 2.
\end{align*}

\subsection*{Overview of new results and related open questions}

In this work, we confirm the above conjecture of Hensley by proving sharp lower bounds on the barycentric coordinates. Another conjecture of Hensley was that
the simplex
\[
 S^d_1 := \conv(\{o, s_1 e_1, \ldots, s_{d-1} e_{d-1}, 2(s_d-1)e_d\}) \in \Sd{1}
\]
contains the maximal number of integral points among all elements of $\Sd{1}$.
This simplex appeared in \cite{MR651251}. Hensley derived a bound on $|S \cap \Z^d|$ for $S \in \Sd{1}$ by 
applying Blichfeldt's theorem (see Theorem~\ref{thm_blichfeldt}), which bounds $|S \cap \Z^d|$ in terms of $\vol(S)$. 
It is therefore natural to modify Hensley's conjecture to the following question: Does $S^d_1$ have maximal volume among all elements of $\Sd{1}$?
Following Hensley's results, Lagarias and Ziegler
\cite{MR1138580}, Pikhurko \cite{MR1996360} and Averkov \cite{MR2967480}  made improvements to the upper bound on the volume of elements in $\Sd{1}$, but 
the bounds they obtained were still different from the largest known examples, the simplices $S^d_1$. For large $d$, the best previously known bound
is due to Pikhurko: he proved that for $S \in \Sd{1}$, one has $\vol(S) \le \frac{1}{d!}2^{3d-2} \cdot 15^{(d-1)2^{d+1}}$. For comparison, one has
$\vol(S^d_1) = \frac{1}{d!}2(s_d - 1)^2 \le \frac{1}{d!} 2^{2^d + 1}$.

We confirm that $S^d_1$ has indeed maximal volume in $\Sd{1}$. We also prove several further results on the size of the elements of $\Sd{1}$. 
In this work, the notions of size we will use are the maximal volume, the maximal and minimal volume of $i$-dimensional
faces for given $i$ and the lattice diameter. We will also give a sharp lower bound on the barycentric coordinates of 
the single interior integral point in an element of $\Sd{1}$, which gives us a bound on the coefficient of asymmetry for elements of $\Sd{1}$. 
By Mahler's theorem, one can bound the volume of a simplex in $\Sd{1}$ in terms of the coefficient of asymmetry. The number of integral points of a
simplex can then be bounded in terms of its volume. This is basically the line of argumentation used in \cite{MR688412}, \cite{MR1138580}, \cite{MR1996360}
and \cite{MR2967480}.
For $S \in \Sd{1}$ and given $l \in \{1, \ldots, d\}$, we give sharp upper bounds on the volume
of every $l$-dimensional face of $S$, in particular for $l=d$ on the simplex $S$ itself. For applications in the geometry
of toric varieties, it is also of interest to analyze the dual
of a simplex in $\Sd{1}$. To this end, for $l \in \{1, \ldots, d\}$ we also bound the volume of $l$-dimensional faces of the dual $S^{\ast}$ of $S \in \Sd{1}$
as well as the Mahler volume $\vol(S)\vol(S^{\ast})$.
We then show that the bound on the coefficient of asymmetry for elements of $\Sd{1}$ remains valid for the more general class $\Pd{1}$. 
This also translates into a volume bound for $\Pd{1}$,
improving the one given in \cite{MR1996360}. For general $k \ge 1$, we determine the sharp upper bound on the lattice diameter of polytopes in $\Pd{k}$.
Finally, we consider lattice-free polyhedra. Such polyhedra have applications in mixed-integer optimization. We extend our result on the lattice diameter
to integral polytopes which are maximal with respect to being lattice-free.

In \cite{MR688412}, \cite{MR1138580} and \cite{MR1996360}, not only elements
of $\Pd{1}$ and $\Sd{1}$ were analyzed, but of $\Pd{k}$ and $\Sd{k}$, respectively, for every $k \ge 1$. 
The special case $k=1$ addressed in this manuscript is of interest from a 
number of perspectives. The relevance of the class $\Sd{1}$ for algebraic geometry is explained in more detail in Section~\ref{sec:alg-geo}, where we present two results 
concerning this topic. In particular, our results have been used to classify all $4$-dimensional weighted projective spaces with canonical singularities
\rescite{fake2}{Theorem 3.3}.
Another topic where elements of $\Sd{1}$ are of interest is mixed-integer optimization. We explain more about
possible applications of our results in this area in Section~\ref{sec:opt_appl}.
Nonetheless, one can of course ask whether our results can be extended to the case $k \ge 1$. First, one can extend the question
about the simplices with maximal volume by considering
the following simplices, which were first introduced in \cite{MR651251}.
For $d \ge 2$ and $k \ge 0$, the $d$-dimensional integral simplex
\begin{align}
S^d_{k} := \conv(\{o, s_1 e_1, \ldots, s_{d-1} e_{d-1}, (k+1)(s_d - 1)e_d\}) \label{eq:zpw_simplex}
\end{align}
contains exactly $k$ interior integral points and has volume $\vol(S^d_k) = \frac{(k+1)}{d!}(s_d - 1)^2$. 
It would be interesting to know whether, analogously to the case $k=1$ discussed in this paper, $S^d_k$ has maximal volume among all elements of $\Sd{k}$ 
also for general $k \ge 2$. 
A straightforward generalization of this question to the case $k=0$ is not possible, as it is easy to see
that there exist integral polytopes without interior integral points and arbitrarily large volume. Indeed, there exist
integral simplices of dimension $d$ with arbitrarily large volume in the slab $[0,1] \times \R^{d-1}$.
One can, however, show that there exist only finitely many integral polytopes without interior integral points under certain assumptions of maximality; 
see \cite{MR2855866}, \cite{MR2832401} and \cite{MR1116368}. Again, this is done by bounding the volume. The known bounds, however,
are far from the largest known examples and hence probably far from the truth.

\subsection*{Organization of the paper and proof ideas}\label{ssc:organization}

The new results are formulated in Section~\ref{sec:results}. In Section~\ref{sec:background}, we collect tools from other sources which we use to prove our results. 
The common approach of the mentioned previous publications on volume bounds for $\Sd{k}$ is to bound the barycentric coordinates
$\beta_1, \ldots, \beta_{d+1}$
of an integral point contained in the interior of a simplex $S \in \Sd{k}$ from below.
For $k=1$, we generalize and improve these bounds.
Let $S \in \Sd{1}$ and let $\beta_1 \ge \ldots \ge \beta_{d+1} > 0$ be the barycentric coordinates of the interior integral point of $S$.
Results from \cite{MR2967480} show that, on the one hand, the inequalities
\[
 \beta_1 \cdots \beta_j \le \beta_{j+1} + \ldots + \beta_{d+1}
\]
hold for every $j \in \{1, \ldots, d\}$ and, on the other hand, the volumes of faces of $S$ can be bounded from above in terms of
values $\frac{1}{\beta_a \cdots \beta_b}$ with $a,b \in \{1, \ldots, d+1\}$ and $a < b$. Thus, if one views $\beta_1, \ldots, \beta_{d+1}$
as arbitrary non-negative variables satisfying $\beta_1 + \ldots + \beta_{d+1} = 1$ and $\beta_1 \cdots \beta_j \le \beta_{j+1} + \ldots + \beta_{d+1}$
for every $j \in \{1, \ldots, d\}$, the determination of the minimum of $\beta_a \cdots \beta_b$ (for given $1 \le a \le b \le d$) allows to find upper
bounds on the volumes of faces of $S$. This purely analytical problem is treated in the first 
part of Section~\ref{sec:auxiliary}. The second part of Section~\ref{sec:auxiliary} is concerned with uniqueness in the special case of minimizing a single barycentric coordinate.
We link this problem to the problem of unit partitions from number theory and modify a result on unit partitions to derive uniqueness. 
In Sections~\ref{sec:simplices_proofs}--\ref{sec:ag}, we then make use of the information about barycentric coordinates obtained in Section~\ref{sec:auxiliary}
and translate it into the size bounds for integral polytopes given in the results from Section~\ref{sec:results}. Furthermore, we characterize the equality cases where possible.
These characterizations are carried out directly using number-theoretical properties of the Sylvester sequence.

\begin{acknowledgments*}\emph{
The authors thank the anonymous referee for numerous very helpful comments and suggestions. 
The second author was supported by a scholarship of the state of Sachsen-Anhalt, Germany.
The third author was supported by the US NSF grant DMS 1203162.} 
\end{acknowledgments*}

\section{Main results}\label{sec:results}

\subsection*{Basic notation and terminology}

For background knowledge from geometry of numbers and convex geometry,
we refer to \cite{MR1940576}, \cite{MR893813}, \cite{MR1434478}, \cite{MR1451876} and \cite{MR1216521}.

Throughout the text, $d \in \N$ is the dimension of the ambient space $\R^d$, which is equipped with the standard scalar product,
denoted by $\sprod{\cdot}{\cdot}$. By $o$ we denote the zero vector, $\allone$ denotes the all-one vector and $e_i$ denotes the $i$-th unit vector. In what follows the dimension of $o$, $\mathbbm{1}$ and $e_i$ is given by the context.
For two points $x,y \in \R^d$, we denote the closed line segment
connecting those points by $[x,y]$ and its Euclidean length by $\len([x,y])$.
For a subset $X$ of $\R^d$, we use $\conv(X)$ and $\lspan(X)$ to denote the \emph{convex hull} and
the \emph{linear hull} of $X$, respectively. The cardinality
of a set $X$ is denoted by $|X|$.

A \emph{polytope} $P \subseteq \R^d$ is the intersection of finitely many closed halfspaces such that $P$ is bounded.
Throughout this paper, we will make use of the following notation: $\relbd{P}$, $\bd{P}$, $\relintr{P}$  and $\intr{P}$ denote the \emph{relative boundary},
\emph{boundary}, \emph{relative interior} and \emph{interior} of a polytope $P$, 
respectively. The set of vertices of $P$ is denoted by $\vertset{P}$. 
We use the notation $\cF_i(P)$ for the set of all $i$-dimensional faces of a $d$-dimensional
polytope $P$, where $i \in \{1, \ldots, d\}$.
For a compact convex set $K \subseteq \R^d$ which contains the origin in its interior, $K^{\ast}$ denotes the {\em polar} (or {\em dual}) body of $K$, i.e.
$K^{\ast} = \{y \in \R^d: \langle x, y \rangle \le 1 \, \forall \, x \in K \}$.
For a polytope $P$ with $o \in \intr{P}$, the polar body $P^{\ast}$ is again a polytope containing $o$ in its interior.
For a compact convex set $K$ with nonempty interior and $u \in \R^d$, we denote by $\rho(K,u)$ the \emph{radius function}  
$\rho(K,u) := \max\setcond{\rho \ge 0}{\rho u \in K}$. Furthermore, $h(K,u)$ denotes the \emph{support function}
$h(K,u) := \operatorname{sup}\setcond{\sprod{u}{x}}{x \in K}$. 
For a $d$-dimensional convex set $K$ and
an integral point $x$ in the interior of $K$, the \emph{coefficient of asymmetry} of $K$ with respect to $x$ is defined 
as
\[
 \ca(K,x) := \max\setcond{\frac{\len([x,a])}{\len([x,b])}}{a,b \in \bd{K}, x \in [a,b]}.
\]

The two different normalizations of volume used in this paper are the standard $i$-dimensional 
\emph{volume} of an $i$-dimensional polytope $P$ ($i \in \{1, \ldots, d\}$),
denoted by $\vol(P)$, 
and for rational polytopes, the \emph{normalized volume}, denoted by $\vol_{\Z}(P)$, which is defined as follows: for a rational polytope $P$ of dimension $ i \in \{1, \ldots, d\}$,
denote by $\Lambda$ the rank $i$ lattice given as the intersection of $\Z^d$ with the linear hull of all vectors $x-y$ such that $x,y \in P$. 
Then $\vol_{\Z}(P)$ is defined as
$\vol(P)/\det(\Lambda)$, where $\det(\Lambda)$ denotes the determinant of $\Lambda$. Clearly, if $P$ is $d$-dimensional,
$\vol_{\Z}(P) = \vol(P)$. Note that in the literature the notion {\em normalized volume} is sometimes used to denote $\dim(P)!\; \vol_{\Z}(P)$.

We say that two integral polytopes $P,Q$ are {\em unimodularly equivalent}, denoted by $P \cong Q$,  
if there exists an affine bijection $\phi$ on $\R^d$ satisfying $\phi(\Z^d) = \Z^d$ and $\phi(P)=Q$.
We call such a bijection a \emph{unimodular transformation}. 
We say that a convex set $C \subseteq \R^d$ is \emph{lattice-free} if it contains
no points of $\Z^d$ in its interior. 
The lattice-free $d$-dimensional integral polytopes are exactly the elements of $\Pd{0}$.
For integral lattice-free polytopes, we define the following two kinds of maximality:
by $\Plm$ we denote the family of all $P \in \Pd{0}$ such that $P+g$ is not lattice-free for every rational line $g$ in $\R^d$ passing 
through the origin (where ``lmax'' is short for ``lineality space maximal''\footnote{The lineality space of a convex body $K \subseteq \R^d$ is the
set of all $u \in \R^d$ such that $K + \lspan(\{u\}) = K$; see \rescite{MR1451876}{p. 65}}) and
by $\Pim$ we denote the family of all $P \in \Pd{0}$ such that $P$ is not properly contained in $Q$ for every $Q \in \Pd{0}$ (where ``imax'' is short for ``inclusion maximal'').
By $\ld{P}$ we denote the \emph{lattice diameter} of a polyhedron $P$, which is defined as the maximum of
$|P \cap \Z^d \cap g| - 1$ over all rational lines $g$. 
To extend this definition to the empty set, we set $\ld{\emptyset} = -1$ in accordance with $|\emptyset \cap \Z^d \cap g| - 1 = -1$.

A $d$-dimensional polytope with $d+1$ vertices is called \emph{simplex} of dimension $d$.
Given a $d$-dimensional simplex $S \subseteq \R^d$ with vertices $v_1, \ldots, v_{d+1}$ and a point $x \in \R^d$, the \emph{barycentric coordinates} of 
$x$ with respect to $S$ are uniquely determined real numbers $\beta_1, \ldots, \beta_{d+1}$ satisfying $\sum_{i=1}^{d+1} \beta_i =1$ 
such that $x = \sum_{i=1}^{d+1} \beta_i v_i$. In this case, for $i \in \{1, \ldots, d+1\}$, we say that $\beta_i$ is the barycentric coordinate associated with $v_i$.
Note that $x \in \intr{S}$ if and only if $\beta_i > 0$ for all $i \in \{1, \ldots, d+1\}$.

\subsection*{Bounds for $\Sd{1}$}\label{ssc:simplices_results}

Our first main result gives sharp lower bounds for the sorted sequence of the barycentric coordinates.
We introduce the following simplices for which those bounds are attained.
For $j \in \{1, \ldots, d+1\}$, we define the simplex $T^d_{1,j} \in \Sd{1}$ by
\[
 T^d_{1,j} := \conv(\{o, s_1 e_1, \ldots, s_{j-1} e_{j-1}, (d-j+2)(s_j-1) e_j, \ldots, (d-j+2)(s_j-1) e_d\}).
\]
Note that in the degenerate cases $j \in \{1, d+1\}$, this definition should be interpreted as
$T^d_{1,1} := \conv(\{o, (d+1) e_1, \ldots, (d+1) e_d\})$ and 
$T^d_{1,d+1} := \conv(\{o, s_1 e_1, \ldots, s_d e_d\})$, respectively. Also, note that $T_{1,d}^d = S_1^d$.
A brief argument why the simplices $T^d_{1,j}$ are in $\Sd{1}$ is given in Remark~\ref{rem_sylv_sim}.

\begin{theorem}\label{main_bc_bound}
Let $S \in \Sd{1}$, where $d \in \N$, and let $i \in \{1, \ldots, d+1\}$. 
Let $\beta_1, \ldots, \beta_{d+1}$ be the barycentric coordinates of the integral point in the interior of $S$ such that 
$\beta_1 \ge \ldots \ge \beta_{d+1}$. Then
\begin{align}
 \beta_i \ge \frac{1}{(d-i+2) (s_i - 1)}. \label{eq:bc_bound_formula}
\end{align}
Furthermore, the following statements hold:
\begin{itemize}
\item[(a)] For $S \cong T^d_{1,i}$, the inequality \eqref{eq:bc_bound_formula} is attained with equality.
 \item[(b)] Inequality \eqref{eq:bc_bound_formula} holds with equality if and only if 
\begin{align*}
 (\beta_1, \ldots, \beta_{d+1}) = \left(\frac{1}{s_1}, \ldots, \frac{1}{s_{i-1}}, \frac{1}{(d-i+2) (s_i - 1)}, \ldots, \frac{1}{(d-i+2) (s_i - 1)}\right).
\end{align*}
\item[(c)] For $i=d+1$, equality holds in \eqref{eq:bc_bound_formula} if and only if $S \cong T^d_{1,i}$.
\end{itemize}
\end{theorem}

Note that assertion (c) cannot be extended to the case $i < d+1$; see Remark~\ref{rem_not_unique_bc}.

This result confirms Hensley's conjecture regarding the coefficient of asymmetry. To see this, observe that 
if for some $S \in \Sd{1}$, we have
$\intr{S} \cap \Z^d = \{x\}$ and $\beta_1 \ge \ldots \ge \beta_{d+1}$ are the barycentric coordinates of $x$ with respect to $S$, the
coefficient of asymmetry can be written as
\begin{align}
 \ca(S,x) = \max_{i \in \{1, \ldots, d+1\}} \frac{1 - \beta_i}{\beta_i} = \frac{1 - \beta_{d+1}}{\beta_{d+1}}. \label{eq:proving_hensley}
\end{align}
This simple fact can be found in \rescite{MR1996360}{(3)}.
Therefore, Theorem~\ref{main_bc_bound}(c) confirms Hensley's conjecture, as $\ca(S,x)$ is maximal when
$\beta_{d+1}$ is minimal.

Next, we
present a sharp upper bound on the face volumes of simplices in $\Sd{1}$, including the $d$-dimensional face, i.e., the simplex itself. 

\begin{theorem}\label{main_volume}
 Let $S \in \Sd{1}$, where $d \ge 3$, and let $l \in \{1, \ldots, d\}$. Then 
\begin{align}
 \max_{F \in \cF_l(S)} \vol_{\Z}(F) \le \frac{2(s_d-1)^2}{l!(s_{d-l+1}-1)}. \label{eq:face_vol_bound}
\end{align}
Furthermore, the following statements hold:
\begin{itemize}
 \item [(a)] For $S \cong S^d_1$, inequality \eqref{eq:face_vol_bound} is attained with equality.
\item [(b)] For $d \ge 4$ and $l \in \{1,d\}$, the equality in \eqref{eq:face_vol_bound} is attained if and only if $S \cong S^d_1$.
\end{itemize}
\end{theorem}

In Theorem~\ref{main_volume}, for $l=d$ we have
\begin{align*}
\vol(S) \le \frac{2}{d!}(s_d-1)^2. 
\end{align*}
This gives us the sharp bound on the volume of the simplices in $\Sd{1}$.
Theorem~\ref{main_volume}
is a generalization of Theorems A and B in \cite{ben}, which are concerned with reflexive simplices, i.e., integral simplices containing the origin and
such that their duals are again integral. For reflexive simplices, Theorems A and B in \cite{ben} cover Theorem~\ref{main_volume}(a) with $l=d$ and 
Theorem~\ref{main_volume}(b). 
Note that while Theorem~\ref{main_volume}(b) guarantees uniqueness for $l \in \{1,d\}$, there is nothing said about $l \in \{2, \dots, d-1\}$.
We ask about a possible characterization of the simplices $S \in \Sd{1}$ for which \eqref{eq:face_vol_bound} is attained with equality in the case 
$l \in \{2, \ldots, d-1\}$.
It has been conjectured that for $d \ge 4$ the inequalities in Theorem~\ref{main_volume}(a) remain valid for arbitrary polytopes in
$\Pd{1}$; see \cite[Conjecture~1.7]{ben}. 
Moreover, one can ask whether in this case equality in all of the inequalities of Theorem~\ref{main_volume}(a) for $l \in \{1, \ldots, d\}$ 
is only satisfied for $S^d_1$, up to unimodular transformation.

As a consequence of Theorem~\ref{main_volume} and a well-known theorem by Blichfeldt (see Theorem~\ref{thm_blichfeldt} below), we derive an upper bound on the number of integral points of
a simplex in $\Sd{1}$ and each of its $l$-dimensional facets.

\begin{corollary}\label{applying_blichfeldt}
  Let $S \in \Sd{1}$, where $d \ge 3$, and let $l \in \{1, \ldots, d\}$. 
Then 
\[
 \max_{F \in \cF_l(S)} |F \cap \Z^d| \le l + l! \vol_{\Z}(F) \le l + \frac{2(s_d-1)^2}{(s_{d-l+1}-1)}.
\]
\end{corollary}

Note that this corollary yields an improvement on the previously known bounds on $|S \cap \Z^d|$. It is, however, not fulfilled
with equality for $S^d_1$. Thus, Hensley's conjecture that $S^d_1$ maximizes $|S \cap \Z^d|$ among all elements in $\Sd{1}$ remains open.

\subsection*{Bounds for $\Sd{1}$ involving dualization}

In the following results, we consider $S \in \Sd{1}$ satisfying $\intr{S} \cap \Z^d = \{o\}$ and provide volume bounds for $S^{\ast}$.
Note that the first of these results is again an extension of a result in \cite{ben}, namely, assertion 2 of Corollary 6.1. It is a sharp estimate on the so-called Mahler volume.

\begin{theorem}\label{mahler} 
Let $d \in \N$ and $S \in \Sd{1}$ and $o \in \intr{S}$. Then 
\[(d+1)^{d+1} \leq (d!)^2\vol(S) \, \vol(S^{\ast}) \leq (s_{d+1} - 1)^2.\]
Furthermore, the lower bound is attained with equality if and only if the unique interior integral point of $S$ equals its centroid, while 
the upper bound is attained with equality if and only if $S \cong T^d_{1,d+1}$.
\end{theorem}

The following result is about face sizes of the dual of a simplex $S \in \Sd{1}$.

\begin{theorem}\label{thm_dual_face_volumes}
 Let $S \in \Sd{1}$, where $d \ge 4$. Let $o \in \intr{S}$ and $l \in \{1, \ldots, d\}$. Then
\begin{align}
 \max_{F \in \cF_l(S^{\ast})} \vol_{\Z}(F) \le \frac{2(s_d-1)^2}{l!(s_{d-l+1}-1)}. \label{eq:dual_vol}
\end{align}
Furthermore, with $T := T^{d-1}_{1,d}$, the following statements hold:
\begin{itemize}
 \item[(a)] Equality holds in \eqref{eq:dual_vol} if $S \cong \conv ((T \times \{0\}) \cup \{\pm e_d\})$.
 \item[(b)] If $l \in \{1,d\}$, then equality holds in \eqref{eq:dual_vol} if and only if $S \cong \conv ((T \times \{0\}) \cup \{\pm e_d\})$.
\end{itemize}
\end{theorem}

We remark that for a simplex $S \in \Sd{1}$, where $d \ge 4$, which satisfies $o \in \intr{S}$ 
and $S \cong \conv ((T \times \{0\}) \cup \{\pm e_d\})$, we have $S^{\ast} \cong S_1^d$.

\subsection*{Results on $\Pd{1}$}\label{ssc:ac_results}

Using Theorem~\ref{main_bc_bound}, we can give a sharp bound on the coefficient of asymmetry for polytopes containing exactly one interior integral point.

\begin{theorem}\label{thm_ca}
 Let $P \in \Pd{1}$, where $d \in \N$, and $o \in \intr{P}$. Then
\[
 \ca(P,o) \le s_{d+1} - 2.
\]
Furthermore, equality holds if and only if $P \cong T^d_{1,d+1}$.
\end{theorem}

This shows that the bound on the coefficient of asymmetry conjectured by Hensley for $\Sd{1}$ remains valid for $\Pd{1}$.
From Theorem~\ref{thm_ca}, we derive a bound on the volume of an arbitrary integral polytope $P$ which has exactly one integral point in its interior.

\begin{theorem}\label{thm_vol_by_ca}
 Let $P \in \Pd{1}$. Then
\[
 \vol(P) \le (s_{d+1} - 1)^d.
\]
\end{theorem}

Asymptotically, Theorem~\ref{thm_vol_by_ca} is an improvement of the known best bound $\vol(P) = 2^{O(d4^d)}$ of Pikhurko \cite{MR1996360}
to the bound $\vol(P) = 2^{O(d2^d)}$.

\subsection*{Bounds on the lattice diameter of integral polytopes}\label{ssc:diameter_results}

We now turn our attention to $\Pd{k}$ for $k \ge 1$ and deduce upper bounds on the lattice diameter $\ld{P}$ for $P \in \Pd{k}$.

\begin{theorem}\label{general_ld_bound}
Let $d \in \N$ and let $P \subseteq \R^d$ be a $d$-dimensional integral polytope such that 
$P' := \conv \bigl(\intr{P} \cap \Z^d \bigr) \neq \emptyset$. Let $m := \ld{P'}$. Then
\[
 \ld{P} \le (m + 2) (s_d - 1).
\]
Furthermore, equality holds if and only if $P \cong S^d_{m+1}$.
\end{theorem}

For $P$ as in Theorem~\ref{general_ld_bound}, consider $G(P):= |P \cap \Z^d|$, the so-called \emph{lattice-point enumerator} of $P$. 
This can be bounded in terms of the lattice diameter, as $\ld{P} + 1 \le G(P) \le (\ld{P}+1)^d$\label{lpe_inequalities}. The lower bound follows from the definitions of $\ld{P}$
and $G(P)$, while the upper bound can be proved using the so-called parity argument; see \rescite{MR1030772}{Theorem 1}.
Thus, in view of Theorem~\ref{general_ld_bound}, the following inequalities hold:
\begin{align}
G(P)^{1/d} -1 \le \ld{P} \le (\ld{P'} + 2) (s_d-1) \le (G(P') + 1) (s_d-1), \label{eq:chain_of_ineq}
\end{align}
where $P'$ is defined as in Theorem~\ref{general_ld_bound}. Thus, each of the two values $G(P)$ and $\ld{P}$ can be bounded in terms
of each of the two values $G(P')$, $\ld{P'}$.
Note that for $m = \ld{P'}$, the interior points of $S^d_{m+1}$ are collinear, that is, for $P=S^d_{m+1}$ one has $\ld{P'}+1 = G(P')$. 
Hence, the equality $\ld{P} = (G(P') + 1) (s_d-1)$ is attained if and only if $P \cong S^d_{m+1}$, as follows from the characterization of the equality case 
in Theorem~\ref{general_ld_bound}.
In contrast to this, the inequalities linking $G(P)$ to $G(P')$ and $\ld{P'}$, respectively, are most likely not tight.

Theorem~\ref{general_ld_bound} can be carried over to polytopes in $\Plm$. The bound we obtain is similar to the one obtained
in Theorem~\ref{general_ld_bound} in the sense that $\ld{P'} + 2 = 1$
if $P$ does not contain any interior integral points. First, observe that if, for some rational line $g$ passing through the origin, 
the polyhedron $P + g$ is lattice-free, then no finite upper bound on the lattice diameter of $P$ can given, because the lattice diameter of 
$P+g$ is infinite and $P$ can be extended to a lattice-free polytope of an arbitrarily large lattice diameter by `adding' sufficiently many integral points from 
$P+g$. On the other hand, the following theorem shows that in all other cases there exists a bound on $\ld P$, which depends only on $d$.

\begin{theorem}\label{lpf_ld_bound}
Let $P \in \Plm$. 
Then the lattice diameter of $P$ is at most $s_d - 1$.
Furthermore, the equality $\ld{P} = s_d-1$ holds if and only if $P \cong S^d_0$.
\end{theorem}

Similarly to Theorem~\ref{general_ld_bound}, also for a polytope $P$ as defined in Theorem~\ref{lpf_ld_bound} we deduce a bound on the lattice point enumerator $G(P)$. 
One has $G(P) \le (s_d)^d \le 2^{d 2^{d-1}}$.
As a consequence, one can deduce a bound on the volume of polytopes in $\Plm$.

\begin{corollary}\label{cor_vol_lpf}
 Let $P \in \Plm$, where $d \in \N$. Then
\[
 \vol(P) \le 2^{2^{2 d +o(d)}}.
\]
\end{corollary}

For $\Pim$ the above bound was shown in \rescite{MR2855866}{Remark 3.10}. Note that $\Pim$ is a proper subset of $\Plm$ for every $d \ge 3$, as for instance 
the $d$-dimensional integral simplex $\conv(\{o, de_1, \ldots, de_{d-1}, (d-1)e_d\})$ is in $\Plm \setminus \Pim$. 
One can see, however, that the corresponding proof in \cite{MR2855866} can be applied for elements of $\Plm$ without any changes. Thus,
Corollary~\ref{cor_vol_lpf} is basically a result from \cite{MR2855866}. A bound of the same asymptotical order 
for all elements of $\Plm$ was also given in \rescite{MR2832401}{Theorem 2.1}; in their proof Nill and Ziegler rely on the results of Kannan and Lov\'{a}sz
\cite{MR970611}. 

\subsection*{Applications of results on $\Sd{1}$ in toric geometry}\label{sec:alg-geo}

This short excursion is intended for readers who are already familiar with algebraic and 
in particular toric geometry. We refer to the book \cite{toricbook} for a general background in toric geometry, and to \cite{gorst, ben} for more on toric Fano varieties and their singularities.

Fano varieties are objects of intense study in algebraic geometry. It is known that $d$-dimensional toric Fano varieties $X$ with 
$\varepsilon$-log terminal singularities form a bounded family; see \cite{MR1166957}. Therefore, it would be very interesting to find sharp upper bounds on their anticanonical degrees $(-K_X)^d$. In \cite{ben} $\Q$-factorial Gorenstein toric Fano varieties with Picard number $1$ were studied in detail. 
They correspond to {\em reflexive simplices} $S$, i.e., integral
simplices $S \in \Sd{1}$ with $S^\ast \in \Sd{1}$.
Recall that Gorenstein toric Fano varieties have canonical singularities; see \cite{MR1269718}. In this paper we extend the main results in \cite{ben} to $\Q$-factorial toric Fano varieties with Picard number $1$ and at most canonical singularities. Such varieties correspond one-to-one (up to 
isomorphisms) to elements of $\Sd{1}$, hence, the combinatorial results of this paper can be applied.

Let us note that $\Q$-factorial toric Fano varieties with Picard number $1$ are also called \textit{fake weighted projective spaces}; see  \cite{fake, MR2549542}. 
They are quotients of weighted projective spaces $\P(q_0, \ldots, q_d)$ by the action of a finite abelian group. 
In particular, the following two theorems hold for weighted projective spaces with canonical singularities.

\begin{theorem}
Let $X$ be a $d$-dimensional $\Q$-factorial toric Fano variety with Picard number one and at most canonical singularities.
\begin{itemize}
\item[(a)] If $d=2$, then
\[(-K_X)^2 \leq 9,\]
with equality if and only if $X \cong \P^2$. 
\item[(b)] If $d=3$, then
\[(-K_X)^3 \leq 72,\]
with equality if and only if $X \cong \P(3,1,1,1)$ or $X \cong \P(6,4,1,1)$.
\item[(c)] If $d \geq 4$, then
\[(-K_X)^d \leq 2 (s_d-1)^2,\]
with equality if and only if $X \cong \P\left(\frac{2 (s_d-1)}{s_1}, \ldots, \frac{2 (s_d-1)}{s_{d-1}}, 1, 1\right)$.
\end{itemize}
\label{ag1}
\end{theorem}

This extends Theorem~A in \cite{ben} to the non-Gorenstein situation. We expect these results to hold for Fano varieties with canonical singularities quite in general. 
We can also generalize Theorem~B in \cite{ben}. 

\begin{theorem}
Let $X$ be a $d$-dimensional $\Q$-factorial canonical toric Fano variety with Picard number one and at most canonical singularities.
Let $\mathcal{C}$ be the set of all torus-invariant integral curves on $X$. Then 
\[\max_{C \in \mathcal{C}} \; (-K_X).C \leq 2 (s_d-1),\]
with equality if and only if $X \cong \P\left(\frac{2 (s_d-1)}{s_1}, \ldots, \frac{2 (s_d-1)}{s_{d-1}}, 1, 1\right)$.
\label{ag2}
\end{theorem}

\subsection*{Applications of $\Pim$ and $\Plm$ to integer optimization}\label{sec:opt_appl}

One application of integral polytopes with a fixed number of interior integral points concerns the so-called reverse Chv\'atal-Gomory rank $\rcg(P)$ of 
integral polytopes $P$ recently introduced in \cite{arxiv:math/1211.0388}. 
The value $\rcg(P)$ is the least upper bound on the Chv\'atal-Gomory rank of rational polyhedra $Q$ satisfying $Q \cap \Z^d = P \cap \Z^d$
(for the definition of and background information on the Chv\'atal-Gomory rank see, 
for example, \rescite{MR874114}{Chapter 23}). 
In \cite{arxiv:math/1211.0388} it is shown that $\rcg(P) < \infty$ if and only if $\relintr{P} \cap \Z^d \neq \emptyset$ 
or $P \in \Plm$. In the case of $\relintr{P} \cap \Z^d \neq \emptyset$, it would be interesting to investigate the relation of $\rcg(P)$ to $|\relintr{P} \cap \Z^d|$
and $\dim(P)$. In this respect, for the simplest possible case $|\relintr{P} \cap \Z^d| = 1$, our theorems about $\Sd{1}$ and $\Pd{1}$ may be interesting.
Regarding $\rcg(P)$ for $P \in \Plm$, our Theorem~\ref{lpf_ld_bound} may be useful.
%

Another topic in optimization related to our studies concerns the use of lattice-free sets for generation of cutting planes for integer and mixed-integer linear programs. The problem of mixed-integer programming consists,
roughly speaking, of optimizing a linear function $f$ over $Q \cap (\Z^d \times \R^n)$, where $Q$ is a rational polyhedron given by a system of linear inequalities, $d \in \N$ and $n \in \N \cup \{0\}$. 
Here, $d$ is the number of integral variables and $n$ the number of real variables. 
The theory of cutting planes arises from the observation that we can make the polyhedron $Q$ smaller
without changing the optimum of $f$ on $Q \cap (\Z^d \times \R^n)$ by replacing $Q$ with $Q \cap H$ as soon as $H$ is a closed halfspace satisfying 
\begin{align} 	
	\label{cut:preserves:miproblem}
	Q \cap H \cap (\Z^d \times \R^n) 
		= 
	Q \cap (\Z^d \times \R^n).
\end{align}
 The idea of cutting-plane methods is that by iteratively reducing $Q$ with appropriately chosen hyperplanes $H$, the optimum of the objective function $f$ on $Q$ will become close enough (or even equal) 
to the optimum of $f$ on $Q \cap (\Z^d \times \R^n)$. 

From the algorithmic perspective, two aspects play a crucial role: on the one hand one needs to be able to construct $H$ efficiently and, 
on the other hand, the effect of application of $H$ for reduction of $Q$ should be significant enough. It is known since the work of Balas 
\cite{Balas71} that halfspaces $H$ as above (the so-called cuts) can be generated using lattice-free sets. If $C$ is a $d$-dimensional convex 
lattice-free set in $\R^d$, then every $H$ satisfying 
\begin{align}
	\label{C-cut condition}
	Q \setminus (\intr{C} \times \R^n) \subseteq H	
\end{align}
 also satisfies \eqref{cut:preserves:miproblem}. The advantage of using \eqref{C-cut condition} instead of \eqref{cut:preserves:miproblem}
for determination of $H$ is that for many choices of $C$ the task of choosing a sufficiently good cut $H$ satisfying \eqref{C-cut condition} can be performed efficiently. 
Note that classical cutting-plane methods rely on very simple lattice-free sets, which are slabs bounded by two parallel hyperplanes (the so-called split sets). However,
recently, it has become clear that the cuts arising from the lattice-free sets in $\Pim$ play an important role in the mixed-integer linear programming. For example, all these cuts are needed for having a finite convergence of the underlying cutting-plane methods; for further details see \cite{MR2855866} and \cite{MR2968262}. This shows that our Theorem~\ref{lpf_ld_bound} might have applications in 
the cutting-plane theory of mixed-integer programs.

\section{Background results}\label{sec:background}

We will make use of the following well-known theorems.

\begin{theorem}\label{thm_blichfeldt}\thmtitle{Blichfeldt's theorem; \cite{MR1500976}, \rescite{MR1434478}{p. 69}}
Let $C \subseteq \R^d$ be a $d$-dimensional convex set. Then
\[
 |C \cap \Z^d| \le d + d!\vol(C).
\]
\end{theorem}

\begin{theorem}\label{thm_minkowski}\thmtitle{Minkowski's first theorem; \rescite{MR1940576}{VII. 3.2}}
 Let $C \subseteq \R^d$ be a $d$-dimensional compact convex set, symmetric up to reflection in $o$ and satisfying $\intr{C} \cap \Z^d = \{o\}$. Then $\vol(C) \le 2^d$.
\end{theorem}

We will also use the following generalization to non-symmetric convex sets, which is due to Mahler.

\begin{theorem}\label{thm_mahler}\thmtitle{Mahler's theorem; \cite{MR0001242}, \rescite{MR893813}{\S 7, Theorem 2}}
Let $C \subseteq \R^d$ be a $d$-dimensional compact convex set and let $\intr{C} \cap \Z^d = \{o\}$. Then $\vol(C) \le (1 + \ca(C,o))^d$.
\end{theorem}

We will frequently use the Sylvester sequence, as defined in the introduction. 
For this sequence the following properties are easy to see.
 
\begin{proposition}\label{lem_syl_seq} \thmtitle{\cite[Lemma 2.1]{MR1138580}, \cite[p. 44]{MR651251}}
For the Sylvester sequence $(s_i)_{i \in \N}$ and every $i \in \N$, the following statements hold:
\begin{compactitem}
\item [(a)] \[s_{i+1} = s_{i}^2 - s_{i} + 1.\] 
\item [(b)] \begin{equation*}\frac{1}{s_{1}} + \ldots + \frac{1}{s_{i}} + \frac{1}{s_{i+1}-1} = 1. \end{equation*}
\item [(c)]
\[
 2^{2^{i-2}} \le s_i \le 2^{2^{i-1}}
\]
and, furthermore, for $i \ge 2$,
\[
 2^{2^{i-2}} + 1\le s_i.
\]
\end{compactitem}
\end{proposition}

\begin{remark}\label{rem_sylv_sim}
In view of Proposition~\ref{lem_syl_seq}, one can now verify in a straightforward way that the simplices $T^d_{1,j}$ are indeed in $\Sd{1}$ for $j \in \{1, \ldots, d+1\}$.
The set $\intr{T^d_{1,j}} \cap \Z^d$ is the set of points $(n_1, \ldots, n_d) \in \N^d$ fulfilling
\begin{align}
\frac{n_1}{s_1} + \ldots + \frac{n_{j-1}}{s_{j-1}} + \frac{n_j + \ldots + n_d}{(d-j+2)(s_j - 1)} < 1. \label{eq:ver_of_t-simplices}
\end{align}
If for $(n_1, \ldots, n_d) \in \N^d$ one has $n_1 = \ldots = n_d = 1$, then \eqref{eq:ver_of_t-simplices}
holds in view of Proposition~\ref{lem_syl_seq}(b). Otherwise, one of the components $n_1, \ldots, n_d$ is at least $2$. Taking into account that
$s_1, \ldots, s_{j-1},(d-j+2)(s_j - 1)$ is an increasing sequence, one sees that for $(n_1, \ldots, n_d) \in \N^d \setminus \{\allone\}$
the smallest value of the left hand side of \eqref{eq:ver_of_t-simplices} is attained for $n_1 = \ldots = n_{d-1} = 1$ and $n_d = 2$. With this choice, the left hand side
of \eqref{eq:ver_of_t-simplices} is 
\[
 \frac{1}{s_1} + \ldots + \frac{1}{s_{j-1}} + \frac{d - j + 2}{(d - j + 2)(s_j - 1)} = 1.
\]
Hence, \eqref{eq:ver_of_t-simplices} is not fulfilled for $(n_1, \ldots, n_d) \in \N^d \setminus \{\allone\}$.
\end{remark}
%
%
%
%

The following two theorems from \cite{MR2967480} deal with simplices $S \in \Sd{1}$. Let $p$ be the unique interior integral point of $S$ and let 
$\beta_1, \ldots, \beta_{d+1}$ be the barycentric coordinates of $p$ with respect to $S$.
It is clear that $\beta_i > 0$ for every $i \in \{1, \ldots, d+1\}$ because $p$ is in the interior of $S$.
Our analysis of $S$ relies mainly on two different results from \cite{MR2967480}.
The first one is a general statement
establishing a system of inequalities for $\beta_1, \ldots, \beta_{d+1}$.

\begin{theorem}\label{thm_sum_prod_bc} \thmtitle{\cite[Theorem 1.1]{MR2967480}}
Let $S \in \Sd{1}$. Let $\beta_1 \ge \ldots \ge \beta_{d+1} > 0$ be the 
barycentric coordinates of the unique interior integral point of $S$. 
Then, for every $j \in \{1, \ldots, d\}$,
\begin{align}
 \beta_1 \cdots \beta_j \le \beta_{j+1} + \ldots + \beta_{d+1}. \label{eq_bc_ineq}
\end{align}
\end{theorem}

Theorem 1.1 in \textup{\cite{MR2967480}} has a somewhat different but equivalent formulation; see the discussion
in \rescite{MR2967480}{page 7}.
The second result from \cite{MR2967480} which we use here links the barycentric coordinates to the face volumes of the simplex. 

\begin{theorem}\label{thm_faces_bc}\thmtitle{\cite[Theorem 3.7]{MR2967480}}
 Let $S \subseteq \R^d$ be a $d$-dimensional simplex with $\vertset{S} \subseteq \Q^d$ containing precisely one interior integral point, which we denote by $p$. 
Let $\beta_1 \ge \ldots \ge \beta_{d+1} > 0$ denote the 
barycentric coordinates of $p$ with respect to $S$. 
Let $F$ be an $l$-dimensional face of $S$, where $l \in \{1, \ldots, d\}$. Let
$\beta_{i_1}, \ldots, \beta_{i_{l+1}}$ be the barycentric coordinates of $p$ associated with the vertices of $F$, where $1 \le i_1 < \ldots < i_{l+1} \le d+1$. Then
\begin{align*}
 \vol_{\Z}(F) \le \frac{1}{l! \beta_{i_1} \cdots \beta_{i_{l}}}.
\end{align*}
\end{theorem}

Direct application of Theorem~\ref{thm_faces_bc} yields
\begin{align}
 \max_{F \in \cF_l(S)}\vol_{\Z}(F) \le \frac{1}{l! \beta_{d-l+1} \cdots \beta_d} \label{eq:face_vs_bc_max}
\end{align}
and
\begin{align}
\min_{F \in \cF_l(S)}\vol_{\Z}(F) \le \frac{1}{l! \beta_1 \cdots \beta_l}, \label{eq:face_vs_bc_min}
\end{align}
where $l \in \{1, \ldots, d\}$.

Note that Theorem~\ref{thm_faces_bc} was only formulated in \cite{MR2967480} for the case of an integral simplex. One can see, however, that
the proof from \cite{MR2967480} can be applied to our more general setting without any changes.
Also, for $l=d$, this result can be found in \rescite{MR1996360}{Lemma 5}; see also \rescite{MR688412}{Theorem 3.4} and 
\rescite{MR1138580}{Lemma 2.3}.

\section{Bounds on barycentric coordinates}\label{sec:auxiliary}

\subsection*{Izhboldin-Kurliandchik type problems}

Because of \eqref{eq:face_vs_bc_max} 
we are interested in the minimal value the product $\beta_a \cdots \beta_{d}$ can attain for any given $a \in \{1, \ldots, d\}$. 
In this section, we will treat a more general problem which we call `Izhboldin-Kurliandchik' problem.
First, we introduce some notation which will be used frequently throughout. For the sake of brevity, let $n := d+1$.
Let $\Chi^n$ denote the set of $n$-tuples $(x_1, \ldots, x_n) \in \R^n$ which fulfil the conditions
\begin{empheq}{align}
& x_1 + \ldots + x_n = 1, &\; \label{cond_sum_one}\\
& 1 \ge x_1 \ge \ldots \ge x_n \ge 0, &\;\label{cond_ordered}\\
& x_1 \cdots x_j \le x_{j+1} + \ldots + x_n \quad \forall \; j \in \{1, \ldots, n-1\}. &\; \label{cond_ps_ineq}
\end{empheq}
Theorem~\ref{thm_sum_prod_bc} shows that every decreasingly sorted $(d+1)$-tuple of barycentric coordinates of the interior integral point of a simplex from $\Sd{1}$
belongs to the set $\Chi^n$. Throughout this section, we shall assume $x_0 := 1$ and $x_{n+1} := 0$.
Furthermore, we introduce additional notation to keep the presentation simple: $\ORD(j)$ denotes the inequality $x_j \ge x_{j+1}$ 
for $j \in \{0, \ldots, n\}$ and will be called the $j$-th \emph{ordering inequality}. 
Similarly, $\PS(j)$ denotes the inequality $x_1 \cdots x_j \le x_{j+1} + \ldots + x_n$ and will be called the $j$-th \emph{product-sum inequality} 
for $x \in \R^n$ and $j \in \{1, \ldots, n-1\}$.

Given a continuous function $f: \Chi^n \rightarrow \R$, we denote by $\IK^n(f)$ the optimization problem of minimizing $f$
over the set $\Chi^n$. This notation is short for `Izhboldin-Kurliandchik problem', as 
Izhboldin and Kurliandchik determined the optimal value and the unique optimal solution of $\IK^n(x_1 \cdots x_n)$ as well as $\IK^n(x_n)$; see \cite{IK_original} and \cite{MR1363298}. For $\IK^n(f)$,
we shall use the standard optimization terminology such as optimal value, optimal solution and feasible solution.

Furthermore, we introduce the following notation to denote special elements of $\Chi^n$: for $l \in \{1, \ldots, n\}$, 
we define the vector $\xquer(l) \in \R^n$ by
\[
 \xquer(l) := \left(\frac{1}{s_1}, \ldots, \frac{1}{s_{l-1}}, \frac{1}{(n-l+1)(s_l - 1)}, \ldots, \frac{1}{(n-l+1)(s_l - 1)} \right).
\]
For $l \in \{1,n\}$, the degenerate cases should be interpreted as $\xquer(1) = \left(\frac{1}{n}, \ldots, \frac{1}{n}\right)$ and $\xquer(n) = \left(\frac{1}{s_1}, \ldots, \frac{1}{s_{n-1}}, \frac{1}{s_n - 1}\right)$.
We can then define the set $\Y$ for every $n \ge 2$ as
\[ 
\Y := \setcond{\xquer(l)}{l = 1, \ldots, n}.
\] 
One can check that every vector $\xquer(l)$ fulfils conditions \eqref{cond_sum_one}--\eqref{cond_ps_ineq} and
thus, $\Y \subseteq \Chi^n$.

\subsection*{Problem $\IK^n(x_a \cdots x_b)$ for general $a$ and $b$}

We modify the
arguments of Izhboldin and Kurliandchik given in \cite{IK_original}
and \cite{MR1363298} to give a more general result about the relation of $\Y$ to the set of optimal solutions of $\IK^n(x_a \cdots x_b)$,
which we state in Lemma~\ref{i_and_k}. In Theorem~\ref{thm_location}, we then show which elements of $\Y$ are optimal solutions of $\IK^n(x_a \cdots x_b)$.

Clearly, $\Chi^n$ is a compact set. It turns out that for every $x \in \Chi^n$ one has $x_1 < 1$ and $x_n > 0$ 
for every $n$.
This and two other basic but important properties, which characterize equality cases for the inequalities describing $\Chi^n$,
are proven in the following lemma. Assertion (c) also establishes the link to the Sylvester sequence.

\begin{lemma}\label{lem_simult_equality}\textup{(Basic properties of $\Chi^n$.)}
 Let $x \in \Chi^n$, where $n \ge 3$. Then the following statements hold:
\begin{itemize}
 \item[(a)] $0 < x_n$ and $1 > x_1$.
 \item[(b)] For $l \in \{1, \ldots, n-1\}$, the inequalities $\ORD(l)$ and $\PS(l)$ for $x$ cannot simultaneously be fulfilled with equality.
 \item[(c)] If for some $l \in \{1, \ldots, n-1\}$, the inequality $\PS(i)$ for $x$ is fulfilled with equality for every $i \in \{1, \ldots, l\}$, then $x_i = \frac{1}{s_i}$ for every $i \in \{1, \ldots, l\}$.
\end{itemize}
\end{lemma}

\begin{proof}
To prove assertion (a), observe that if $x_n = 0$,
then by $\PS(n-1)$ also $x_1 \cdots x_{n-1} = 0$ and the ordering inequalities then imply $x_{n-1} = 0$. Iteratively, the above argument leads to $x_i = 0$ for 
every $i \in \{1, \ldots, n\}$, yielding a contradiction to $x_1 + \ldots + x_n = 1$. If $x_1 = 1$, 
then in view of $x_1 + \ldots + x_n = 1$, we get $x_2 = \ldots = x_n = 0$, a contradiction to $\PS(1)$.

We show assertion (b) by proving that if for some $l \in \{1, \ldots, n-1\}$ we have $x_1 \cdots x_l = x_{l+1} + \ldots + x_n$, this implies $x_l > x_{l+1}$. 
Since $n \ge 3$, the product $x_1 \cdots x_l$ contains more than one factor or the sum $x_{l+1} + \cdots + x_n$ contains more than one summand. Hence, in
view of (a), 
\[
 x_l \ge x_1 \cdots x_l = x_{l+1} + \ldots + x_n \ge x_{l+1},
\]
where at least one of the two inequalities is strict. Thus, we obtain $x_l > x_{l+1}$.

Assertion (c) follows by induction on $i$. For $i = 1$, from $x_1 = 1 - x_1$ we get $x_1 = 1/2 = 1/s_1$. Assuming the statement holds up to some $i \in \{1, \ldots, l-1\}$,
we get from the definition of the Sylvester sequence, Proposition~\ref{lem_syl_seq} and $x_1 \cdots x_i x_{i+1} = 1 - (x_1 + \ldots + x_i + x_{i+1})$:
\[
 \frac{x_{i+1}}{s_{i+1}-1} = \frac{1}{s_{i+1}-1} - x_{i+1}
\]
and hence $x_{i+1} = \frac{1}{s_{i+1}}$.
\end{proof}

The aim of this section is to find optimal solutions for $\IK^n(x_a \cdots x_b)$, where $1 \le a \le b \le n$.
In the following lemma, we deal with the more general problem $\IK^n(x_1^{\alpha_1} \cdots x_n^{\alpha_n})$ for real numbers $\alpha_1, \ldots, \alpha_n$. 
Using a limit argument, we can then derive assertions about $\IK^n(x_a \cdots x_b)$.

\begin{lemma}\label{i_and_k} \thmtitle{On optimal solutions of $\IK^n(x_1^{\alpha_1} \cdots x_n^{\alpha_n})$}
  Let $n \ge 2$, $b \in \{1, \ldots, n\}$ and $\alpha_1, \ldots, \alpha_n \in \R$ be such that $0 \le \alpha_i \le \alpha_j$ for all $1 \le i \le j \le b$
and $\alpha_i = \alpha_j \le 0$ for all $i,j > b$. Then the following assertions hold:
\begin{itemize}
 \item [(a)] There exists an optimal solution $x^{\ast}$ for $\IK^n(x_1^{\alpha_1} \cdots x_n^{\alpha_n})$ with $x^{\ast} \in \Y$.
 \item [(b)] If $\alpha_i \neq 0$ for all $i \in \{1, \ldots, n\}$, then all optimal solutions of $\IK^n(x_1^{\alpha_1} \cdots x_n^{\alpha_n})$ are in $\Y$.
 \item [(c)] If $b=n-1$, $\alpha_i > 0$ for all $i \neq n$ and $\alpha_n = 0$, then all optimal solutions of $\IK^n(x_1^{\alpha_1} \cdots x_n^{\alpha_n})$ are in $\Y$.
 \item [(d)] If $a = 1$, $b = n$ and $\alpha_i > 0$ for all $i \in \{1, \ldots, n\}$, then the unique optimal solution to $\IK^n(x_1^{\alpha_1} \cdots x_n^{\alpha_n})$ is
\[
 \left(\frac{1}{s_1}, \ldots, \frac{1}{s_{n-1}}, \frac{1}{s_n - 1} \right).
\]
\end{itemize}
\end{lemma}

\begin{proof}
The proof is organized as follows. We start by proving an auxiliary claim. Then we prove (b) and subsequently deduce (a) as a consequence of (b). Next,
we observe that the proof of (b) can be slightly modified to give a proof of (c). Finally, we prove (d) by showing that under the additional assumptions,
the arguments used in the proof of (b) allow an explicit description of the optimal solution.

\begin{claim}\label{claim_ik}
If $b < n$ and $\alpha_i < 0$ for all $i \in \{b+1, \ldots, n\}$, then every optimal solution $x^{\ast}$ of $\IK^n(x_1^{\alpha_1} \cdots x_n^{\alpha_n})$ has the 
property
\begin{align}
 x_b^{\ast} = \ldots = x_n^{\ast}. \label{last_few_equal}
\end{align}
\end{claim}

\begin{proof}[Proof of the claim]\renewcommand{\qedsymbol}{$\blacksquare$}
We assume the contrary, i.e., the existence of an optimal solution $x^{\ast} \in \Chi^n$ which does not satisfy \eqref{last_few_equal}. Let $j \in \{b, \ldots, n-1\}$ 
be the maximal index such that 
$\ORD(j)$ is strict for $x^{\ast}$. This allows us to construct an
element $x' \in \Chi^n$ by subtracting a small $\delta > 0$ from the $j$-th component of $x^{\ast}$ and adding $\delta$ to the $(j+1)$-th component.
To see that $x'$ is indeed an element of $\Chi^n$, observe that its components still sum to one and all ordering inequalities are still fulfilled.
It remains to check that $\PS(i)$ remains valid for every $i$. For $i < j$, this is obviously the case: neither the product part nor
the sum part of the inequalities $\PS(1), \ldots, \PS(j-1)$ are affected. $\PS(j)$ also remains valid as its product part becomes smaller while its sum part becomes
larger. For $i > j$, by maximality of $j$,
we have that $\ORD(i)$ is fulfilled with equality for $x^{\ast}$. So, by Lemma~\ref{lem_simult_equality}(b), the inequality $\PS(i)$ has to be strict for $x^{\ast}$ 
and hence remains valid if $\delta$ is chosen small enough. Now we show that $x'$ is a better feasible solution of $\IK^n(x_1^{\alpha_1} \cdots x_n^{\alpha_n})$.
First we treat the case $j>b$, for which the values of the objective function at $x^{\ast}$ and $x'$ differ by the factor
\begin{align}
 \frac{(x'_1)^{\alpha_1} \cdots (x'_n)^{\alpha_n}}{(x^{\ast}_1)^{\alpha_1} \cdots (x^{\ast}_n)^{\alpha_n}} = \left(1 - \frac{\delta}{x_j^{\ast}} \right)^{\alpha} \left( 1 + \frac{\delta}{x_{j+1}^{\ast}}\right)^{\alpha} = \left(1 + \frac{\delta}{x_{j+1}^{\ast}} - \frac{\delta}{x_j^{\ast}} - \frac{\delta^2}{x_j^{\ast} x_{j+1}^{\ast}} \right)^{\alpha}, \label{eq:par_eq}
\end{align}
where $\alpha := \alpha_j = \alpha_{j+1}$. 
For a sufficiently small $\delta > 0$, this factor is less than one as $\ORD(j)$ is strict for $x^{\ast}$ and $\alpha < 0$.
This yields the desired contradiction.
If $j=b$, then the contradiction is more easily
obtained, as in this case $\alpha_j \ge 0 > \alpha_{j+1}$. Hence the contradiction follows immediately from $(x_j^{\ast} - \delta)/x_j^{\ast} < 1 < (x_{j+1}^{\ast} + \delta)/x_{j+1}^{\ast}$.
This proves the claim.
\end{proof}

{\em Assertion (b).} Assume that $\alpha_i \neq 0$ for all $i \in \{1, \ldots, n\}$. Let $x^{\ast}$ be an optimal solution of $\IK^n(x_1^{\alpha_1} \cdots x_n^{\alpha_n})$. 
We show $x^{\ast} \in \Y$.
By $k \in \{0, \ldots, n-1\}$ we denote the index 
such that $x_k^{\ast} > x_{k+1}^{\ast} = \ldots = x_n^{\ast}$. The degenerate case $k=0$ is trivial, as then $x^{\ast} = (1/n, \ldots, 1/n) \in \Y$.
We therefore assume $k \ge 1$. Observe that by Claim~\ref{claim_ik}, if $b<n$, then this implies $k \le b-1$ and hence $\alpha_i > 0$ for $i \in \{1, \ldots, k\}$.
We show that, for all $j < k$, the inequality $\ORD(j)$ is strict. In view of Lemma~\ref{lem_simult_equality}(a), the inequality $\ORD(0)$ is strict
and hence for $k=1$, there is nothing to show. It remains to show that for $k \ge 2$, the inequality $\ORD(j)$ is strict for every $j < k$.
Note that $\ORD(k)$ is strict by construction and, as stated above, $\ORD(0)$ is strict as well.
Thus, if there exists $j \in \{1, \ldots, k-1\}$ such that $\ORD(j)$ is fulfilled with equality, one can determine
a pair of indices $i_1, i_2$ with $1 \le i_1 < i_2 \le k$ such that, for $x^{\ast}$, the inequalities 
$\ORD(i_1 - 1)$ and $\ORD(i_2)$ are strict, while
$\ORD(j)$ is fulfilled with equality for $j \in \{i_1, \ldots, i_2 - 1\}$. We can again perturb $x^{\ast}$ to some $x' \in \Chi^n$ by adding a small 
$\delta > 0$ to $x_{i_1}^{\ast}$ and subtracting $\delta$ from $x_{i_2}^{\ast}$. For convenience, we write $y$ for the value of $x_{i_1}^{\ast}$ and $x_{i_2}^{\ast}$,
which are equal by construction. Again, to prove $x' \in \Chi^n$ we only need to look at the product-sum inequalities.
For $j < i_1$, both sides of $\PS(j)$ remain unchanged. For $j \in \{i_1, \ldots, i_2 - 1\}$, Lemma~\ref{lem_simult_equality}(b) yields that $\PS(j)$ is strict for $x^{\ast}$ and
hence remains valid for $x'$ if $\delta > 0$ is small enough. For $j \ge i_2$, only the product part of $\PS(j)$ is affected. To be more precise, the product part changes by the factor $(y + \delta)(y - \delta)/y^2$, which is less than one. 
Thus, $\PS(j)$ is still valid for $x'$ and we conclude that $x' \in \Chi^n$. Passing from $x^{\ast} \in \Chi^n$ to $x' \in \Chi^n$, the value of the
objective function changes by the factor
\[
 \left(\frac{y + \delta}{y}\right)^{\alpha_{i_1}} \left(\frac{y - \delta}{y}\right)^{\alpha_{i_2}} \le \left(1 - \frac{\delta^2}{y^2} \right)^{\alpha_{i_2}} < 1
\]
since $\alpha_{i_1} \le \alpha_{i_2}$.
This is a contradiction to the choice of $x^{\ast}$. We deduce that $\ORD(j)$ is strict for every $j < k$. Next, we show that for $x^{\ast} \in \Chi^n$ and 
every $1 \le j < k$, the inequality $\PS(j)$ is fulfilled with equality. We assume the contrary. Then $k \ge 2$ and there exists some index $i \in \{1, \ldots, k-1\}$ such that $\PS(i)$ is strict.
Since $\ORD(i-1)$ and $\ORD(i+1)$ are strict, we can again add a $\delta > 0$ to $x_i^{\ast}$ and subtract $\delta$ from $x_{i+1}^{\ast}$ to obtain a new element $x'$ of $\Chi^n$
by the same arguments as above. This is a contradiction since 
the values of the objective function $x_1^{\alpha_1} \cdots x_n^{\alpha_n}$ at $x'$ and at $x^{\ast}$
differ by the factor
\begin{align}
 \frac{(x_i^{\ast} + \delta)^{\alpha_i}(x_{i+1}^{\ast} - \delta)^{\alpha_{i+1}}}{x_i^{*\alpha_i}x_{i+1}^{*\alpha_{i+1}}} &= \left(1 + \frac{\delta}{x_i^{\ast}}\right)^{\alpha_i}\left(1 - \frac{\delta}{x_{i+1}^{\ast}}\right)^{\alpha_{i+1}} \notag \\
& \le \left(1 + \frac{\delta}{x_i^{\ast}} - \frac{\delta}{x_{i+1}^{\ast}} - \frac{\delta^2}{x_i^{\ast} x_{i+1}^{\ast}}\right)^{\alpha_{i+1}}. \label{eq:ps_still_valid}
\end{align}
For a sufficiently small $\delta > 0$, this factor is less than one since $\alpha_{i+1} > 0$. This is a contradiction to the choice of $x^{\ast}$.

We show now that $\PS(k)$ is also fulfilled with equality. Suppose
that this is not the case, i.e., $x^{\ast}_1 \cdots x^{\ast}_k < x^{\ast}_{k+1} + \ldots + x^{\ast}_n$. Note that we have $x^{\ast}_{k-1} > x^{\ast}_k > x^{\ast}_{k+1} = \ldots = x^{\ast}_n$. 
By Lemma~\ref{lem_simult_equality}(b), this implies that for $k+1 \le j \le n-1$, the inequality $\PS(j)$ is strict for $x^{\ast}$. Let $\delta$ be close to $0$, but not necessarily positive. We define $x(\delta) := (x^{\ast}_1, \ldots, x^{\ast}_{k-1}, x^{\ast}_k + \delta(n-k), x^{\ast}_{k+1} - \delta, \ldots, x^{\ast}_n - \delta)$.
Observe that, if $\delta$ is sufficiently close to $0$, $x(\delta) \in \Chi^n$. To see this, note that $\ORD(j)$ is unaffected for $1 \le j < k-1$, while $\ORD(k-1)$ and $\ORD(k)$ remain valid for $x$ if $\delta$ is sufficiently
close to $0$. For $k+1 \le j \le n-1$, both sides of $\ORD(j)$ change by the same amount. $\PS(j)$ obviously remains valid for $1 \le j \le k-1$. Since for $k \le j \le n-1$,
the inequality $\PS(j)$ is strict for $x^{\ast}$ it remains valid for $x$ if $\delta$ is sufficiently close to $0$.

We consider the function $f$ in $\delta$ given by
\[
 f(\delta) = (x_1^{\ast})^{\alpha_1} \cdots (x_{k-1}^{\ast})^{\alpha_{k-1}} (x_k^{\ast} + \delta (n-k))^{\alpha_k} (x_{k+1}^{\ast} - \delta)^{\alpha_k} \cdots (x_n^{\ast} - \delta)^{\alpha_n}.
\]
We want to show that $f$ does not have a local minimum in $\delta = 0$ and hence, $x^{\ast} = x(0)$ cannot be an optimal solution
of $\IK^n(x_1^{\alpha_1} \cdots x_n^{\alpha_n})$.  As $f(0)$
is fixed (and not $0$), it is more convenient to analyze the scaled function $g(\delta) = f(\delta) / f(0)$. For convenience,
we also write $y := x^{\ast}_{k+1} = \ldots = x^{\ast}_n$ and $\beta = \alpha_{k+1} + \ldots + \alpha_n$. Using this notation, we have
\[
 g(\delta) = \left( 1 + \frac{\delta(n-k)}{x^{\ast}_k} \right)^{\alpha_k} \left( 1- \frac{\delta}{y}\right)^{\beta}.
\]
The first derivative at $\delta = 0$ is
\[
 g'(0) = \frac{\alpha_k (n-k)}{x^{\ast}_k} - \frac{\beta}{y}.
\]
If $g'(0) \neq 0$, $0$ is not a local minimum of $g$. If $g'(0) = 0$, we have $\alpha_k (n-k) y = \beta x^{\ast}_k$. 
Let us assume that this holds, as otherwise, we are done.
Note that since $\alpha_k$, $n-k$, $y$ and $x^{\ast}_k$ are all positive, this implies $\beta > 0$.
We want to show that if $g'(0) = 0$, then $0$ is strict local maximum of $g$. Hence, we need to show $g''(0) < 0$.
We now consider the second derivative at $\delta = 0$, which is
\[
 g''(0) = \alpha_k (\alpha_k - 1) \frac{(n-k)^2}{(x^{\ast}_k)^2} - 2 \frac{\alpha_k (n-k) \beta}{x^{\ast}_k y} + \frac{\beta (\beta - 1)}{y^2}.
\]
In view of the assumption $\alpha_k (n-k) y = \beta x^{\ast}_k$, we get
\[
 g''(0) = \frac{\beta^2}{y^2} - \frac{(n-k)\beta}{x^{\ast}_k y}  - 2 \frac{\beta^2}{y^2} + \frac{\beta^2 - \beta}{y^2}.
\]
As $\beta > 0$ and $y > 0$, for $g''(0) < 0$ it suffices to have
\[
 - \frac{(n-k)}{x^{\ast}_k} - \frac{1}{y} < 0. 
\]
Since $n-k$, $x^{\ast}_k$ and $y$ are positive, this statement is true. Hence, $x^{\ast}$ cannot be an optimal solution 
of $\IK^n(x_1^{\alpha_1} \cdots x_n^{\alpha_n})$, which is a contradiction to the choice of $x^{\ast}$.
Thus, we deduce that $\PS(k)$ is fulfilled with equality.

Hence, we have that for every optimal solution $x^{\ast}$ under the assumptions that $\alpha_i \neq 0$ for every $i \in \{1, \ldots, n\}$, inequality $\PS(j)$ is satisfied with equality 
for every $1 \le j \le k$. By 
Lemma~\ref{lem_simult_equality}(c), we get that $x_j^{\ast} = 1/s_j$ for $j \in \{1, \ldots, k\}$. Since $x_j$ has been uniquely determined for $j \in \{1, \ldots, k\}$,
the value for $x_{k+1}^{\ast} = \ldots = x_n^{\ast}$ can be uniquely determined from $x_1 + \ldots + x_n = 1$. In view of Proposition~\ref{lem_syl_seq}(b), we get 
\[
 x_{k+1}^{\ast} = \ldots = x_n^{\ast} = \frac{1}{(n-k)(s_{k+1} - 1)}
\]
and hence $x^{\ast} \in \Y$, thereby proving (b).

{\em Assertion (a).} We now drop the assumption $\alpha_i \neq 0$ for all $i \in \{1, \ldots, n\}$. Note that since for each $i \in \{1, \ldots, n\}$
and every $x \in \Chi^n$ we have $x_i > 0$, the expression $x_i^{\alpha_i}$ is continuous in $\alpha_i$. As a consequence,
for every fixed $x \in \Chi^n$, the expression $x_1^{\alpha_1} \cdots x_n^{\alpha_n}$ is continuous in $(\alpha_1, \ldots, \alpha_n)$.
Hence, an optimal solution can be
deduced from (b) by a limit argument. For $t \in \N$, we introduce the function 
$f_{t}(x) := x_1^{\alpha_1 + \frac{1}{t}} \cdots x_b^{\alpha_b + \frac{1}{t}} x_{b+1}^{\alpha_{b+1} - \frac{1}{t}} \cdots x_n^{\alpha_n - \frac{1}{t}}$.
By (b), the problem $\IK^n(f_{t})$ has an optimal solution $x_{t}^{\ast}$ belonging to $\Y$, i.e. 
\begin{align}
f_{t}(x_{t}^{\ast}) \le f_{t}(x) \label{eq:limit}
\end{align}
for every $x \in \Chi^n$. 
Since $\Y$ is finite, there exists an element $y \in \Y$ and an infinite set $T \subseteq \N$ such that for every $t \in T$, we have $x_{t}^{\ast} = y$.
Passing to the limit in \eqref{eq:limit} as $t \in T$ tends to $\infty$, we deduce that $y$ is an optimal solution of $\IK^n(x_1^{\alpha_1} \cdots x_n^{\alpha_n})$. This shows (a).

{\em Assertion (c).}
Note that because we pass to the limit, we cannot guarantee that in the case of a general choice of $(\alpha_1, \ldots, \alpha_n)$ all optimal solutions
of $\IK^n(x_1^{\alpha_1} \cdots x_n^{\alpha_n})$ are in $\Y$. In particular, one needs a different argument to prove (c).
We show that if $\alpha_n = 0$ but $\alpha_i > 0$ for $i \neq n$, we have that $x^{\ast}_{n-1} = x^{\ast}_n$ holds for every optimal solution $x^{\ast}$. To see this, assume
the contrary, i.e. $x^{\ast}_{n-1} > x^{\ast}_n$. Then, we can perturb $x^{\ast}$ into some $x'$ by setting $x'_i := x_i^{\ast}$ for $i \in \{1, \ldots, n-2\}$,
$x'_{n-1} := x_{n-1}^{\ast} - \delta$ and $x'_n := x_n^{\ast} + \delta$ for $\delta > 0$ sufficiently small. Obviously, the product-sum inequalities remain valid and thus $x' \in \Chi^n$.
This leads to
\[
(x'_1)^{\alpha_1} \cdots (x'_{n-1})^{\alpha_{n-1}} < (x_1^{\ast})^{\alpha_1} \cdots (x_{n-1}^{\ast})^{\alpha_{n-1}}, 
\]
a contradiction to the choice of $x^{\ast}$. Now we can apply the arguments of
assertion (b) to obtain (c).

{\em Assertion (d).} We
show that if $x^{\ast}$ is an optimal solution, $x^{\ast}$ satisfies $\PS(j)$ with equality for every $j \in \{1, \ldots, n-1\}$.
Let $i_1,i_2 \in \{1, \ldots, n-1\}$ with $i_1 < i_2$ be such that for $x^{\ast}$, the inequality $\PS(j)$ is strict for $i_1, \ldots, i_2$, but
$\PS(i_1 - 1)$ and $\PS(i_2 + 1)$ are fulfilled with equality. Of course, the requirement on $\PS(i_1 - 1)$ is only considered if $i_1 > 1$ and on $\PS(i_2 + 1)$
only if $i_2 < n-1$. 
Then by Lemma~\ref{lem_simult_equality}(b), $\ORD(i_1 -1)$ and $\ORD(i_2 +1)$ are strict. Note that for the degenerate cases $i_1 = 1$ or $i_2 = n-1$
this still holds and hence we do not need to consider them separately.
We consider $x'$ given by $x'_i := x^{\ast}_i$ for $i \in \{1, \ldots, n\} \setminus \{i_1,i_2\}$ and $x'_{i_1} := x^{\ast}_{i_1} + \delta$,
$x'_{i_2 + 1} := x^{\ast}_{i_2 + 1} - \delta$ for some $\delta > 0$. Since $\ORD(i_1 - 1)$ and $\ORD(i_2 + 1)$ are strict, they remain valid for $x'$
if $\delta$ is sufficiently small. To show $x' \in \Chi^n$, we check that $x'$ satisfies all product-sum inequalities. This can be done in the same way as
in the proof of assertion (b). We then argue that $x'$ is a better solution to $\IK^n(x_1^{\alpha_1} \cdots x_n^{\alpha_n})$ by repeating 
the argumentation leading to \eqref{eq:ps_still_valid} with $i_1$ instead of $i$ and $i_2 + 1$ instead of $i+1$. This yields a contradiction and
proves that $x^{\ast}$ satisfies $\PS(j)$ with equality for every $j \in \{1, \ldots, n-1\}$. Applying Lemma~\ref{lem_simult_equality}(c)
then yields $x_j = 1/s_j$ for every $j \in \{1, \ldots, n-1\}$ and from $x^{\ast}_1 + \ldots + x^{\ast}_n = 1$ we obtain $x_n = 1/(s_n - 1)$. 
\end{proof}

We remark that Lemma~\ref{i_and_k}(d) is a generalization of the result proved by Izhboldin and Kurliandchik, who assumed $\alpha_i = 1$ instead
of $\alpha_i > 0$ for every $i \in \{1, \ldots, n\}$.
In the following, we analyze the problem $\IK^n(x_a \cdots x_b)$ for specific choices of $a$ and $b$.
Lemma~\ref{i_and_k} states that for every $a,b$ such that $1 \le a \le b \le n$, we find an optimal solution of $\IK^n(x_a \cdots x_b)$ in $\Y$.
The following theorem now deals with the question
which elements of $\Y$ are optimal solutions of $\IK^n(x_a \cdots x_b)$ depending on the choice of $a$ and $b$.

\begin{theorem}\label{thm_location}\thmtitle{Localization of optimal solutions of $\IK^n(x_a \cdots x_b)$ within $\Y$}
 Let $n,a,b \in \N$ be such that $n \ge 4$ and $1 \le a \le b \le n$. Let $l \in \{1, \ldots, n\}$. Then the following statements hold:
\begin{itemize}
\item[(a)] If $a=b$, then $\xquer(b)$ is an optimal solution of $\IK^n(x_b)$.
\item [(b)] If $b < n-1$, $a < b$ and $\xquer(l)$ is an optimal solution of $\IK^n(x_a \cdots x_b)$, then $a \le l \le \min \{b, 2 + \log_2\log_2 (n e)\}$ or $l = b$.
\item [(c)] If $b=n$, then $\xquer(l)$ is an optimal solution of $\IK^n(x_a \cdots x_n)$ if and only if $l = n$. 
\item [(d)] If $b = n-1$, then $\xquer(l)$ is an optimal solution of $\IK^n(x_a \cdots x_b)$ if and only if $l = n-1$ or
$n=4$, $1 \le a \le 2$, $l=2$.
\item [(e)] If $n \ge 5$, $a=1$ and $b=n-1$, then
\[
 \xquer(n-1) = \left( \frac{1}{s_1}, \cdots, \frac{1}{s_{n-2}}, \frac{1}{2(s_{n-1}-1)}, \frac{1}{2(s_{n-1}-1)}\right)
\]
is the unique optimal solution of $\IK^n(x_1 \cdots x_{n-1})$. If $n=4$ and $a=1$ and $b=3$, then
$\xquer(2) = (1/2, 1/6, 1/6, 1/6)$ and $\xquer(3) = (1/2, 1/3, 1/12, 1/12)$ are the only optimal solutions of $\IK^4(x_1 x_2 x_3)$.
\end{itemize}
\end{theorem}

\begin{proof}
By Lemma~\ref{i_and_k}(a), for every choice of $a,b$ there exists an optimal solution of $\IK^n(x_a \cdots x_b)$ which is in $\Y$
and hence can be expressed as $\xquer(l)$. 
For the proof, we view $l$ as a variable ranging in $\{1, \ldots, n\}$. Let $y(l) = (y(l)_1, \ldots, y(l)_n)$. We introduce the function $f(l) := 1 / (\xquer(l)_a \cdots \xquer(l)_b)$. 
The structure of the proof is as follows: first, we show that $a \le l \le b$
if $l$
maximizes $f(l)$.  
We then show that $f(l)$ is strictly monotonously increasing in $l$
provided $l$ ranges in $\{a, \ldots, b\}$ and is not too close to $a$ or $b$.

\begin{claim}\label{claim_loc_1}
 If $\xquer(l)$ is an optimal solution of $\IK^n(x_a \cdots x_b)$, then $l \in \{a, \ldots, b\}$.
\end{claim}

\begin{proof}[Proof of the claim]\renewcommand{\qedsymbol}{$\blacksquare$}
We first show that if $\xquer(l)$ is an optimal solution of $\IK^n(x_a \cdots x_b)$, then $l \ge a$. 
This is obviously true if $a = 1$. Assume $a > 1$ and $l < a$. We compare $f(l)$ and $f(l+1)$ and want to show the inequality
\[
 f(l) = ((n-l+1)(s_l - 1))^{b-a+1} < ((n-l)(s_{l+1}-1))^{b-a+1} = f(l+1).
\]
Using $s_{l+1} - 1 = s_l(s_l-1)$ we can reformulate the desired inequality as
\[
 (n-l+1)(s_l - 1) < (n-l)(s_l-1)s_l
\]
or, equivalently,
$1 < (s_l - 1)(n-l)$. The latter is clearly true since for $l = 1$, one has $s_l = 2$ and $n - l \ge 3$ while for $l > 1$, one has $s_l \ge 3$
and $n - l \ge 1$. This shows that if $l$ maximizes $f(l)$, we have $l \ge a$.
To conclude the proof of the claim it remains to show that if $f(l)$ is maximized for $l \in \{1, \ldots, n\}$, then $l \le b$. 
We assume $b < n$, because otherwise there is nothing to
show. Observe that by the definition of the Sylvester sequence
\[
f(b) = \frac{(n-b+1)(s_b - 1)^2}{s_a - 1}, 
\]
while for $l > b$,
\[
f(l) = \frac{(s_b - 1)s_b}{s_a - 1}. 
\]
Hence, for $l > b$ we have that $f(b) > f(l)$ holds whenever $(n-b+1)(s_b - 1) > s_b$ or, equivalently, $(n-b)(s_b - 1) > 1$.
If $b \ge 2$, i.e., $s_b - 1 \ge 2$, this holds because $n > b$. If $b = 1$, we have $s_b - 1 = 1$ and since $n \ge 4$,
we have $n - b \ge 3$.
Thus, if$f(l)$ attains the maximal value, then $l \in \{a, \ldots, b\}$.
\end{proof}

{\em Assertion (a).} Claim~\ref{claim_loc_1} shows that if $a = b$, then $l = a = b$. Together with Lemma~\ref{i_and_k}(a), this proves assertion (a). 

\begin{claim}\label{claim_loc_late_add}
 Let $a \le l,l' \le b$. Then $f(l) < f(l')$ if and only if
\[
 (n-l+1)^{b-l+1}(s_l - 1)^{b-l+2} < (n-l'+1)^{b-l'+1}(s_{l'} - 1)^{b-l'+2}.
\]
\end{claim}

\begin{proof}[Proof of the claim]\renewcommand{\qedsymbol}{$\blacksquare$}
 By definition of $f$, the inequality $f(l) < f(l')$ is equivalent to
\[
 ((n-l+1)(s_l - 1))^{b-l+1}\prod_{i=a}^{l-1} s_i  <  ((n-l'+1)(s_{l'} - 1))^{b-l'+1}\prod_{i=a}^{l'-1} s_i. 
\]
Multiplying both sides with $s_1 \cdots s_{a-1}$ and using the definition of the Sylvester sequence proves the claim.
\end{proof}

We determine a subrange of $\{a, \ldots, b\}$ in which $f(l)$ is monotone.

\begin{claim}\label{claim_loc_2}
If $a \le l < b$, $l \ge 2$ and $s_l - e(n-b+2) \ge 0$, where $e$ is the Euler number, then $f(l) < f(l+1)$.
\end{claim}

\begin{proof}[Proof of the claim]\renewcommand{\qedsymbol}{$\blacksquare$}
By Claim~\ref{claim_loc_late_add}, $f(l) < f(l+1)$ is equivalent to
\begin{align}
 (n-l+1)^{b-l+1}(s_l - 1)^{b-l+2} < (n-l)^{b-l}(s_{l+1} - 1)^{b-l+1}. \label{eq:to_show_mon}
\end{align}
Note that $n-l \ge b-l \ge 1$. Using this and $s_{l+1} - 1 = s_l(s_l - 1)$, one sees that \eqref{eq:to_show_mon} is equivalent to
\[
\left(1 + \frac{1}{n-l}\right)^{b-l} (n-l+1) (s_l - 1) < s_l^{b-l+1}.
\]
Since $n-l \ge b-l$,
\[
\left(1 + \frac{1}{n-l}\right)^{b-l} \le \left(1 + \frac{1}{n-l}\right)^{n-l} < e,
\]
where $e$ is the Euler number. Thus, \eqref{eq:to_show_mon} holds if
\[
s_l^{b-l} - e(n-l+1) \ge 0.
\]
We introduce $x := b-l-1$. Then $n - l = n - b + b - l = n - b + x + 1$ and we rewrite the latter inequality as 
\[
s_l^{x+1} - e(x + n - b + 2) \ge 0.
\]
Note that $x \ge 0$.
For a moment, we view $x$ as a variable
ranging in $[0, \infty)$ and develop $s_l^{x+1} - e(n - b + x + 2)$ into its Taylor series at the point $x = 0$:
\[
 s_l^{x+1} - e(x + n - b + 2) = s_l - e(n-b+2) + (s_l \ln s_l - e)x + \sum_{k=2}^{\infty} \frac{s_l (\ln s_l)^k}{k!} x^k.
\]
Obviously, 
\[
\sum_{k=2}^{\infty} \frac{s_l (\ln s_l)^k}{k!}x^k \ge 0 
\]
for $x \ge 0$ as each of the coefficients is positive. The inequality
$s_l \ln s_l - e > 0$ holds whenever $l \ge 2$. Thus, for $l \ge 2$, the expression $s_l^{x+1} - e(n - b + x + 2)$
is monotonously increasing for $ x \in [0, \infty)$. Hence its value
for $x = b-l-1$ is not smaller than the value for $x=0$. Therefore, for \eqref{eq:to_show_mon} to hold it suffices to have
\begin{align}
s_l - e(n-b+2) \ge 0. \label{eq:to_show_ref}
\end{align}
\end{proof}

{\em Assertion (b).} 
Here, we have $1 \le a < b$, i.e. $b \ge 2$ and therefore we also have $n-b+2 \le n$. 
Let $a \le l < b$.
In view of Proposition~\ref{lem_syl_seq}(c), we have  $s_l \ge 2^{2^{l-2}}$.
Assume now that
\[
l > 2 + \log_2 \log_2 (e n).
\]
Combining this with our previous observations yields
\[
 s_l \ge 2^{2^{l-2}} \ge e n \ge e (n - b + 2).
\]
By Claim~\ref{claim_loc_2}, this implies that $f(l) < f(l+1)$ and thus, 
for $2 + \log_2 \log_2 (en) < l < b$, the value $f(l)$ is not the maximum.
This proves (b).

%

{\em Assertion (c).}
We assume $b = n$ and $a < b$. Let $a \le l < b$. If $l \ge 3$,
we have
$s_l - e(n-b+2) = s_l - 2e > 0$ and hence by Claim~\ref{claim_loc_2}, we know that $\xquer(l)$ is not an optimal solution
of $\IK^n(x_a \cdots x_b)$.
Therefore, only $\xquer(1)$, $\xquer(2)$ and $\xquer(n)$ remain as possible optimal solutions.
We show directly that
$\xquer(n)$ is a better solution than $\xquer(1)$ and $\xquer(2)$, respectively.
By Claim~\ref{claim_loc_late_add}, we need to show
\[
 (n + 1 - l)^{n+1-l}(s_l - 1)^{n-l+2} < (s_n - 1)^2
\]
for $n \ge 4$ and $l \in \{1,2\}$. We proceed by induction on $n$. One verifies directly that the inequality holds for
$n = 4$ and $l \in \{1,2\}$. Let us now show that the inequality remains valid when moving from $n$
to $n+1$. By induction hypothesis,
\[
 (s_{n+1} - 1)^2 = s_n^2(s_n - 1)^2 >  (n + 1 - l)^{2 (n + 1 - l)}(s_l - 1)^{2(n-l+2)}.
\]
To complete the induction, we need to show
\begin{align*}
(n + 1 - l)^{2 (n + 1 - l)}(s_l - 1)^{2(n-l+2)} \ge (n + 2 - l)^{n+2-l}(s_l - 1)^{n-l+3}.
\end{align*}
This is equivalent to
\begin{align}
 (n + 1 - l)^{n-l}(s_l - 1)^{n-l+1} \ge \left(1 + \frac{1}{n + 1 - l}\right)^{n+2-l}. \label{eq:to_complete_ind}
\end{align}
The right-hand side of \eqref{eq:to_complete_ind} is at most $e \left(1 + \frac{1}{n + 1 - l}\right)$. As $l \le 2$ and $n \ge 4$,
this is at most $\frac{4e}{3}$. Again employing $l \le 2$ and $n \ge 4$, the left-hand side of \eqref{eq:to_complete_ind}, is at least $(n+1-l)^{n-l} \ge 9 \ge \frac{4e}{3}$.
This completes the induction.
Lemma~\ref{i_and_k}(a) now yields assertion (c).

{\em Assertions (d) and (e).}
We assume $b=n-1$ and $a < b$. Let $a \le l < b$.
If $l \ge 4$, we have
$s_l - e(n - b + 2) = s_l - 3e > 0$ and hence by Claim~\ref{claim_loc_2}, we have that $\xquer(l)$ is not an optimal solution of $\IK^n(x_a \cdots x_b)$.
Thus, it remains to compare $\xquer(l)$ and $\xquer(b) = \xquer(n-1)$ directly for $l \in \{1,2,3\}$. Note that if $a > 1$,
we can exclude the elements $\xquer(l)$ with $l < a$.
By Claim~\ref{claim_loc_late_add}, we need to show
\[
(n + 1 - l)^{n-l+1}(s_l - 1)^{n-l+2} < 2 (s_{n-1} - 1)^2 \qquad \text{for} \; l \in \{1,2,3\}.
\]
Using an inductive argument analogous to the one used for $b=n$, we obtain that this inequality is valid for every $n \ge 5$ and every $l \in \{1,2,3\}$.
If $n=4$ and $l=1$, one can verify directly that the inequality holds. If $n=4$ and $l \in \{2,3\}$, a direct computation shows
that $\xquer(2)$ and $\xquer(3) = \xquer(n-1)$ yield the same values for each of the problems $\IK^4(x_1 x_2 x_3)$ and $\IK^4(x_2 x_3)$.
This proves assertion (d). Lastly, we prove (e) and consider the problem $\IK^n(x_1 \cdots x_{n-1})$ with $n \ge 4$. By Lemma~\ref{i_and_k}(c), 
all optimal solutions of this problem are in $\Y$. Hence by our previous arguments, assertion (e) follows.
\end{proof}

\subsection*{Problem $\IK^n(x_a \cdots x_b)$ for $a=b$}
\label{sec:equality}

One of the well-known topics in number theory (see, e.g., \cite{MR1520110, MR0043117}) concerns the so-called unit partitions, i.e., representations of $1$ in the form
$1 = \frac{1}{u_1} + \ldots + \frac{1}{u_n}$ with $n \in \N$ and $u_1, \ldots, u_n \in \N$.
By symmetry reasons, there is no loss of generality in assuming the sequence $u_1, \ldots, u_n$ to satisfy $u_1 \le \ldots \le u_n$. 
With this assumption the vector $x = (\frac{1}{u_1}, \ldots ,\frac{1}{u_n})$ belongs to $\Chi^n$. To see that all product-sum inequalities are satisfied,  
observe that for every $j \in \{1, \ldots, n-1\}$ we have
\begin{align*}
\frac{1}{u_{j+1}} + \ldots + \frac{1}{u_n} = 1 - \sum_{i=1}^j \frac{1}{u_i}. 
\end{align*}
Writing the expression on the right hand side as a fraction with denominator $u_1 \cdots u_j$, one can see that the numerator is at least $1$,
hence the fraction is at least $\frac{1}{u_1 \cdots u_j}$. Therefore,
\[
 \setcond{\left( \frac{1}{u_1}, \ldots, \frac{1}{u_n}\right)}{u_1, \ldots, u_n \in \N, \; \frac{1}{u_1} + \ldots + \frac{1}{u_n} =1, \; u_1 \le \ldots \le u_n} \subseteq \Chi^n.
\]
Previous work on this subject has established a link between unit partitions and the Sylvester sequence; see \cite{MR1520110} and \cite{MR0043117}.
A result of Soundararajan from (\rescite{arxiv:math/0502247v1}{(1)}) provides an elegant way of establishing a sharp upper bound on $u_j$ for each $j \in \{1, \ldots, n\}$.

Here, we use the ideas from \cite{arxiv:math/0502247v1} to show a generalized result. 
From Theorem~\ref{thm_location}(a), we already know that $\xquer(a)$ is an optimal solution of $\IK^n(x_a)$ and the only optimal solution contained in $\Y$. 
However, Theorem~\ref{thm_location}(a) does not state anything about uniqueness of this solution within the whole set $\Chi^n$. 
Following Soundararajan's approach closely, we can show that $\xquer(a)$ is indeed the unique optimal solution within $\Chi^n$.
The proof strategy from \cite{arxiv:math/0502247v1} can be adapted by dropping the assumption
of unit fractions and replacing it by the less restrictive assumption of product-sum inequalities. 
This will lead to optimal solutions of $\IK^n(x_a)$ for $a \in \{1, \ldots, n\}$
over the set $\Chi^n$.
First, we need the following proposition from \cite{arxiv:math/0502247v1}. It is based on Muirhead's inequality. Muirhead's inequality can be found, for example, in \rescite{MR0046395}{Theorem 45}.

\begin{proposition}\label{prop_sound_use}\thmtitle{\rescite{arxiv:math/0502247v1}{Proposition}}
Let $n \in \N$. Let
$g_1 \geq g_2 \geq \ldots \geq g_n > 0$ and 
$h_1 \geq h_2 \geq \ldots \geq h_n > 0$ be such that 
$h_1 \cdots h_j \leq g_1 \cdots g_j$ for every 
$j \in \{1, \ldots, n\}$. Then 
\begin{equation*}
h_1 + \ldots + h_n \le g_1 + \ldots + g_n, 
\end{equation*} 
with equality if and only if $g_i = h_i$ for every $i \in \{1, \ldots, n\}$.
\end{proposition}


\begin{lemma}\label{prop_sound_rev}
Let $x_1, \ldots, x_k$ ($k \in \N$) be a sequence of positive real numbers satisfying the following conditions:
\begin{empheq}{align}
& x_1 + \ldots + x_k < 1, &\; \label{k_sum_less}\\
& x_1 \ge \ldots \ge x_k, &\;\label{k_ordered}\\
& x_1 \cdots x_j \le 1 - \sum_{i=1}^j x_i \quad \forall \; j \in \{1, \ldots, k\}. &\; \label{k_ps_ineq}
\end{empheq}
Then
\begin{align}
x_1 + \ldots + x_k \le \frac{1}{s_1} + \ldots + \frac{1}{s_k}, \label{sound_ineq_ts}
\end{align}
with equality being attained if and only if $x_i = 1/s_i$ for every $i \in \{1, \ldots, k\}$.
\end{lemma}

\begin{proof} 
The proof of Lemma~\ref{prop_sound_rev} is an adaptation of
the proof presented in \cite{arxiv:math/0502247v1}. Since \cite{arxiv:math/0502247v1} is not intended for publication, we repeat the details
of the proof here.

We argue by induction on $k$. For $k = 1$, inequality
\eqref{k_ps_ineq} yields $x_1 \le 1/2 = 1/s_1$. 
Let now $k \ge 2$ and assume that the statement holds for all sequences with at most $k-1$ elements.
First, we show that \eqref{sound_ineq_ts} holds in the case that 
\begin{align} 
x_1 \cdots x_k > \frac{1}{s_1 \cdots s_k}. \label{eq:sound_aux_label}
\end{align}
In view of \eqref{k_ps_ineq}, we obtain
\begin{align*} 
\frac{1}{s_1} + \ldots + \frac{1}{s_k} = 1 - \frac{1}{s_{k+1} - 1} = 1 - \frac{1}{s_1 \cdots s_k} > 1 - x_1 \cdots x_k \ge \sum_{i=1}^k x_i.
\end{align*}
Hence \eqref{sound_ineq_ts} holds and is strict. We switch to the case $x_1 \cdots x_k \le \frac{1}{s_1 \cdots s_k}$.
Thus, the inequality
\begin{align}
x_l \cdots x_k \le \frac{1}{s_l} \cdots \frac{1}{s_k} \label{eq:ind_step}
\end{align}
holds for $l=1$. We fix $l \in \{1, \ldots, k\}$ to be the maximal index for which \eqref{eq:ind_step} holds. Since $l-1 < k$, the induction hypothesis yields
\begin{align}
\sum_{i=1}^{l-1} x_i \le \sum_{i=1}^{l-1} \frac{1}{s_i}. \label{eq:main1_first}
\end{align}
The choice of $l$ yields a series of inequalities as follows:
\begin{align*}
x_l &\le \frac{1}{s_l}, \\
x_l x_{l+1} &\le \frac{1}{s_l s_{l+1}},\\
&\;\; \vdots \\
x_l \cdots x_k &\le \frac{1}{s_l \cdots s_k}.
\end{align*}
Thus, $x_l, \ldots, x_k$ and $1/s_l, \ldots, 1/s_k$ fulfil the settings of Proposition~\ref{prop_sound_use}. Hence, we have
\begin{align}
\sum_{i=l}^k x_i \le \sum_{i=l}^k \frac{1}{s_i}. \label{eq:main1_second}
\end{align}
Together with \eqref{eq:main1_first}, this proves \eqref{sound_ineq_ts}.

Now, we characterize the equality case. Observe that, if $x_i = 1/s_i$ for every $i \in \{1, \ldots, k\}$, the inequality
\eqref{sound_ineq_ts} is trivially fulfilled with equality. Let us now assume we have equality in \eqref{sound_ineq_ts} and show that this 
implies $x_i = 1/s_i$ for every $i \in \{1, \ldots, k\}$. Above, we have shown that \eqref{eq:sound_aux_label} implies
that \eqref{sound_ineq_ts} is strict. So, under our assumptions, \eqref{eq:sound_aux_label} cannot hold.
We can hence again fix the maximal $l$ such that \eqref{eq:ind_step} holds. Then by the induction hypothesis, we have $x_i = 1/s_i$ for 
every $i \in \{1, \ldots, l-1\}$. Therefore, \eqref{eq:main1_first} holds with equality. This implies that \eqref{eq:main1_second}
also has to be fulfilled with equality. 
By Proposition~\ref{prop_sound_use}, this is the case if and only if $x_i = 1/s_i$ for every $i \in \{l, \ldots, k\}$.
\end{proof}

\begin{theorem}\label{thm_single_low}
Let $n \in \N$, $i \in \{1, \ldots, n\}$ and $x \in \Chi^n$.
Then
\begin{align*}
x_i \ge \frac{1}{(n-i+1)(s_i - 1)},
\end{align*}
with equality if and only if $x = \xquer(i)$.
\end{theorem}

\begin{proof}
We have
\begin{align}
 x_i \ge \frac{1}{n-i+1}\left(1 - \sum_{j=1}^{i-1} x_j\right) \ge \frac{1}{n-i+1}\left( 1 - \sum_{j=1}^{i-1} \frac{1}{s_j} \right) = \frac{1}{(n-i+1)(s_i-1)}. \label{eq:bound_for_one}
\end{align}
The first inequality in \eqref{eq:bound_for_one} follows from $x_1 + \ldots + x_n = 1$ and $x_i \ge x_{i+1} \ge \ldots \ge x_n$ and the second one is a consequence of
Lemma~\ref{prop_sound_rev}. Now, let us prove the characterization of the equality case. Clearly, for $x = \xquer(i)$, one
has $x_i = \frac{1}{(n-i+1)(s_i - 1)}$. On the other hand, if $x_i = \frac{1}{(n-i+1)(s_i - 1)}$, both inequalities in \eqref{eq:bound_for_one}
are attained with equality.
Hence, we have $1 - \sum_{j=1}^{i-1} x_j = 1/(s_i - 1)$, 
which by Lemma~\ref{prop_sound_rev} holds if and only if $x_j = 1/s_j$ for $j \in \{1, \ldots, i-1\}$.
Furthermore, $x_1 + \ldots + x_n = 1$ yields $x_i + \ldots + x_n = 1/(s_i - 1)$. Since we have $x_i \ge x_{i+1} \ge \ldots \ge x_n$ and $x_i = \frac{1}{(n-i+1)(s_i - 1)}$,
this implies that we have in fact 
$x_i = x_{i+1} = \ldots = x_n$. This yields $x = \xquer(i)$.
\end{proof}

\begin{remark}\label{ik_remark}
One can see that Theorem~\ref{thm_single_low} solves $\IK^n(x_i)$ for every $i \in \{1, \ldots, n\}$. 
In \cite{IK_original}, there is an alternative approach to solve $\IK^n(x_n)$. First, Izhboldin and Kurliandchik
determine the unique optimal solution of  $\IK^n(x_1 \cdots x_n)$ and
then show that this is also the unique optimal solution of $\IK^n(x_n)$ in the following way. Let $y$ be the
optimal solution of $\IK^n(x_1 \cdots x_n)$. We know from the proof of Lemma~\ref{i_and_k} that $y$ attains all
the product-sum inequalities with equality. Let $x$ be an arbitrary
element of $\Chi^n$. Thus, taking into account $\PS(n-1)$ for $x$, we have
\begin{align}
x_n^2 \ge x_1 \cdots x_{n-1} x_n \ge  y_1 \cdots y_n = y_n^2.  \label{ik_rem}
\end{align}
Thus, $x_n \ge y_n$, which shows that $y$ is also an optimal solution of
$\IK^n(x_n)$. Furthermore, $y$ is the unique optimal solution of
$\IK^n(x_n)$, because if $x \in \Chi^n$ minimizes $x_n$, then all
inequalities in \eqref{ik_rem} are attained with equality. Consequently, $x$ is an
optimal solution of $\IK^n(x_1 \ldots x_n)$ and by this coincides with
$y$.
\end{remark}

\section{Proofs of results for $\Sd{1}$ (Theorems~\ref{main_bc_bound} and \ref{main_volume})}\label{sec:simplices_proofs}


In this section, we apply the results from Section~\ref{sec:auxiliary} to the $(d+1)$-tuple of barycentric coordinates associated with the interior
integral point of a simplex in $\Sd{1}$.
Observe that if we assume the barycentric coordinates $\beta_1, \ldots, \beta_{d+1}$ of this point
to be ordered decreasingly, the tuple $(\beta_1, \ldots, \beta_{d+1})$ fulfils conditions
\eqref{cond_sum_one} and \eqref{cond_ordered}. Because of Theorem~\ref{thm_sum_prod_bc}, it also fulfils \eqref{cond_ps_ineq} and hence we have 
$(\beta_1, \ldots, \beta_{d+1}) \in \Chi^{d+1}$.

The following simple proposition will be used for the characterization of the equality cases in Theorems~\ref{main_bc_bound} and \ref{main_volume}.
It will also be used in Section~\ref{sec:diameter}, where we prove Theorems~\ref{general_ld_bound} and \ref{lpf_ld_bound}.

\begin{proposition}\label{div_property_sylv}
 Let $a_1, \ldots, a_k$ ($k \in \N$) be pairwise relatively prime integers. Let $m_1, \ldots, m_k$ be integers such that 
\[
 \frac{m_1}{a_1} + \ldots + \frac{m_k}{a_k} \in \Z.
\]
Then $a_i$ divides $m_i$ for every $i \in \{1, \ldots, k\}$.
\end{proposition}

\begin{proof}
By symmetry of the statement, it suffices to prove the assertion for $i=1$. 
Multiplying $\frac{m_1}{a_1} + \ldots + \frac{m_k}{a_k}$ by $a_2 \cdots a_k$, we deduce $\frac{m_1 a_2 \cdots a_k}{a_1} \in \Z$. By the assumptions, 
$a_1$ is relatively prime to $a_2 \cdots a_k$. Hence $\frac{m_1}{a_1} \in \Z$ and the assertion follows.
\end{proof}

It follows directly from the definition that the elements of the Sylvester sequence are pairwise relatively prime. Hence, we can apply Proposition~\ref{div_property_sylv}
to the elements of the Sylvester sequence.

\begin{proof}[Proof of Theorem~\ref{main_bc_bound}]
The bound in \eqref{eq:bc_bound_formula} and assertion (b) follow immediately from Theorem~\ref{thm_single_low} by setting $n = d+1$ and 
$x_1 = \beta_1, \ldots, x_n = \beta_{n}$. Assertion (a) can be checked directly from the fact that $\intr{T^d_{1,i}} \cap \Z^d = \{(1, \ldots, 1)\}$
and by noticing that $\xquer(i) \in \mathcal{Y}^{d+1}$ is the $(d+1)$-tuple of barycentric coordinates of $(1, \ldots, 1)$ with respect to $T^d_{1,i}$.

We will now prove assertion (c). From assertion (a) we have that for $S \cong T^d_{1,d+1}$ and $i = d+1$, equality is attained in \eqref{eq:bc_bound_formula}.
To show the reverse implication, consider an arbitrary $S \in \Sd{1}$ such that its interior integral point $p$ has the barycentric coordinate 
$\beta_{d+1} = \frac{1}{s_{d+1}-1}$.
Assertion (b) implies
\[
 (\beta_1, \ldots, \beta_{d+1}) = \left(\frac{1}{s_1}, \ldots, \frac{1}{s_d}, \frac{1}{s_{d+1}-1}\right).
\]
Without loss of generality, we can assume that the vertex of $S$ associated with the smallest barycentric coordinate is $o$. For the remaining barycentric coordinates,
let $p_i$ with $i \in \{1, \ldots, d\}$ be the vertices of $S$ such that $p_i$ corresponds to the coordinate $1/s_i$, $i \in \{1, \ldots, d\}$. Then
\[
 p := \frac{p_1}{s_1} + \ldots + \frac{p_d}{s_d}
\]
is the single interior integral point of $S$. By Proposition~\ref{div_property_sylv}, we get $p_i/s_i \in \Z^d$ for $i \in \{1, \ldots, d\}$.
Let $\Lambda := \frac{p_1}{s_1}\Z + \ldots + \frac{p_d}{s_d}\Z$ be the lattice induced by those points. By construction, $\Lambda$ is a rank $d$ sublattice of $\Z^d$.
Recall that for a rank $d$ lattice with basis $b_1, \ldots, b_d \in \Z^d$, the determinant of this lattice can be expressed as the number of lattice points in
$(0,1]b_1 + \ldots + (0,1]b_d$ (see, e.g., \rescite{MR1940576}{VII, 2.6}). Using this and the fact that
\begin{align*}
\{p\} & \subseteq \left((0,1]\frac{p_1}{s_1} + \ldots + (0,1]\frac{p_d}{s_d}\right) \cap \Lambda \\
& \subseteq \intr{S} \cap \Lambda \\ & \subseteq \intr{S} \cap \Z^d = \{p\},
\end{align*}
we have that $\Lambda$ is a sublattice of $\Z^d$ with determinant one. Hence $\Lambda = \Z^d$ and, moreover, 
the linear mapping given by $\frac{p_i}{s_i} \mapsto e_i$ preserves $\Z^d$. Consequently, this mapping is a unimodular transformation. Thus
$S \cong \conv(\{o, s_1 e_1, \ldots, s_d e_d\}) = T^d_{1,d+1}$.
\end{proof}

To prove Theorem~\ref{main_volume}(b), we will show that equality in \eqref{eq:face_vol_bound} for $l \in \{1,d\}$ is attained
if and only if the unique interior integral point of the simplex $S$ has a specific tuple of barycentric coordinates with respect to $S$.
The following lemma gives a characterization of the simplices in $\Sd{1}$ for which this is the case.
The same tuple of barycentric coordinates will also be of interest when characterizing the equality case in Theorem~\ref{thm_dual_face_volumes}.

\begin{lemma}\label{lem_hyperplane_intersection}
 Let $S \in \Sd{1}$. Let the barycentric coordinates $\beta_1, \ldots, \beta_{d+1}$ of the interior integral
point with respect to $S$ be
\[
  (\beta_1, \ldots, \beta_{d+1}) = \left(\frac{1}{s_1}, \ldots, \frac{1}{s_{d-1}}, \frac{1}{2(s_d - 1)}, \frac{1}{2(s_d - 1)}\right).
\]
Then $S \cong \conv \left(T^{d-1}_{1,d} \times \{0\} \cup \{\pm(a, h)\} \right)$ for some $h \in \N$ and some $a \in \{0, \ldots, h-1\}^{d-1}$.
\end{lemma}

\begin{proof}
 Let $p_i$ be the vertex associated with the barycentric coordinate $1/s_i$ for $i \in \{1, \ldots, d-1\}$ and $p_d,p_{d+1}$ be the vertices associated with the
remaining two coordinates. Then, the single interior integral point of $S$ is
\[
 p := \frac{p_1}{s_1} + \ldots + \frac{p_{d-1}}{s_{d-1}} + \frac{p_d + p_{d+1}}{2(s_d - 1)}.
\]
Since $1/(s_d - 1) = 1 - 1/s_1 - \ldots - 1/s_{d-1}$, we get
\[
 2 p = p_d + p_{d+1} + \sum_{i=1}^{d-1} \frac{2 p_i - p_d - p_{d+1}}{s_i}.
\]
As $2p, p_d, p_{d+1} \in \Z^d$, we obtain
\[
 \sum_{i=1}^{d-1} \frac{2 p_i - p_d - p_{d+1}}{s_i} \in \Z^d.
\]
Applying Proposition~\ref{div_property_sylv}, we get $(2 p_i - p_d - p_{d+1})/s_i \in \Z^d$ for all $i \in \{1, \ldots, d-1\}$. 
In particular, $(p_d + p_{d+1})/2 \in \Z^d$, because $s_1 = 2$ and $p_1 \in \Z^d$.
Without loss of generality, we assume that $(p_d + p_{d+1})/2 = o$. 
The simplex $T := \conv(\{o, p_1, \ldots, p_{d-1}\})$ is a hyperplane section
of $S$. Furthermore, $p \in \relintr{T}$ and the 
$d$-tuple of barycentric coordinates of $p$ with respect to $T$ is
\[
 \left(\frac{1}{s_1}, \ldots, \frac{1}{s_{d-1}}, \frac{1}{s_d - 1}\right).
\]
Theorem~\ref{main_bc_bound}(c) asserts that (up to unimodular equivalence) there is only one $(d-1)$-dimensional simplex with one interior integral point
which has these coordinates and thus $T \cong T^{d-1}_{1,d} \times \{0\}$. By moving to another basis of $\Z^{d-1} \times \{0\}$ if necessary, we can
assume $T = \conv(\{o, s_1 e_1, \ldots, s_{d-1} e_{d-1}\})$. Then, we can write $p_d$ as $(a,h)$, where $a \in \Z^{d-1}$ and $h \in \Z$. 
 Since $p_d + p_{d+1} = o$, this leads to $p_{d+1} = -(a,h)$. 
Note that $h \neq 0$ because otherwise $S$ would not be full-dimensional. Furthermore, we can assume $h \in \N$ by interchanging the roles
of $p_d$ and $p_{d+1}$ if necessary.
It remains to show that up to unimodular equivalence, we can choose $a \in \Z^{d-1}$ such that $a \in \{0, \ldots, h-1\}^{d-1}$. 
Assume $a \notin \{0, \ldots, h-1\}^{d-1}$. Let $a' \in \{0, \ldots, h-1\}^{d-1}$ 
such that $a' \equiv a \modulo{h}$. There exists
a linear unimodular transformation preserving $\R^{d-1} \times \{0\}$ which maps the simplex $\left(T^{d-1}_{1,d} \times \{0\} \cup \{\pm(a, h)\} \right)$
to the simplex $\left(T^{d-1}_{1,d} \times \{0\} \cup \{\pm(a', h)\} \right)$. This observation completes the proof.
\end{proof}

In the proof of Theorem~\ref{main_volume}(b), we will show that the simplices for which equality in \eqref{eq:face_vol_bound} holds for $l \in \{1,d\}$
fulfil the assumptions of Lemma~\ref{lem_hyperplane_intersection}. We will then proceed by determining the values of $a$ and $h$ as used in the formulation
Lemma~\ref{lem_hyperplane_intersection}. To determine $a$, we will make use of the following lemma.

\begin{lemma}\label{lem_cov_min_simplex}
Let $S$ be the simplex in $\R^d$ given by $S = \conv(\{o, a_1 e_1, \ldots, a_d e_d\})$, where $d \ge 2$ and $a_1, \ldots, a_d > 1$ 
satisfy $\frac{1}{a_1} + \ldots +\frac{1}{a_d} = 1$. 
Let $v \in \R^d$. Then,
$S + v$ is lattice-free if and only if $v \in \Z^d$.
\end{lemma}

\begin{proof}
It is easy to see that $S$ is lattice-free and it is obvious that $S+v$ is also lattice-free for every $v \in \Z^d$.
Let us now show the reverse implication. Let $v \in \R^d  \setminus \Z^d$.
Observe that
\[
 \intr{S} = \setcond{x = (x_1, \ldots, x_d) \in \R^d}{x_1, \ldots, x_d > 0, \; \frac{x_1}{a_1} + \ldots + \frac{x_d}{a_d} < 1}.
\]
Hence, taking into account $\frac{1}{a_1} + \ldots +\frac{1}{a_d} = 1$, we get
$(0,1]^{d} \setminus \{\allone\} \subseteq \intr{S}$.  
Then, $(0,1]^d + v$ contains the point $v' := (\lfloor v_1 + 1 \rfloor, \ldots, \lfloor v_d + 1 \rfloor) \in \Z^d$. 
It is easy to check that since $v \not\in \Z^d$, we have $v' \neq v + \allone$.
Thus, $v' \in (0,1]^{d} \setminus \{\allone\} + v \subseteq \intr{S + v}$ and hence $S + v$ is not lattice-free.
\end{proof}

\begin{proof}[Proof of Theorem~\ref{main_volume}]
The proof is divided into four parts. First, we prove \eqref{eq:face_vol_bound} and assertion (a). Assertion (b) is then proven separately for
the cases $l=1$ and $l=d$.

{\em Inequality \eqref{eq:face_vol_bound}.} Let $\beta_1, \ldots, \beta_{d+1}$ be the barycentric coordinates of the interior integral point of $S$
and assume $\beta_1 \ge \ldots \ge \beta_{d+1}$.
By \eqref{eq:face_vs_bc_max}, we have
\[
 \max_{F \in \cF_l(S)} \vol_{\Z}(F) \le \frac{1}{l! \beta_{d-l+1} \cdots \beta_d}.
\]
Minimizing the product $\beta_{d-l+1} \cdots \beta_{d}$ corresponds to finding an optimal solution to $\IK^{d+1}(x_{d-l+1} \cdots x_d)$. 
Inequality~\eqref{eq:face_vol_bound} then follows from Theorem~\ref{thm_location}(d) with $n = d+1$. 

{\em Assertion (a).}
We need to show that
for every $l \in \{1, \ldots, d\}$, the simplex $S^d_1$ fulfils \eqref{eq:face_vol_bound} with equality. This can be seen from the fact that
for $l \in \{1, \ldots, d\}$, the simplex $S^d_1$ contains the face 
\[
\conv(\{o, s_{d-l+1} e_{d-l+1}, \ldots, s_{d-1} e_{d-1}, 2(s_d - 1)e_d\}),
\] 
which has normalized volume
\[
 \frac{2}{l!} (s_d - 1) \prod_{i = d-l+1}^{d-1} s_{i} = \frac{2(s_d - 1)^2}{l! (s_{d-l+1} - 1)}.
\]

\begin{claim}\label{clm_equiv_of_s1d}
Let $S := \conv \bigl(T^{d-1}_{1,d} \times \{0\} \cup \{\pm(s_d - 1)e_d\}\bigr)$. Then, $S \cong S^d_1$.
\end{claim}

\begin{proof}[Proof of the claim]\renewcommand{\qedsymbol}{$\blacksquare$}
Note that because $(s_d - 1)e_d \in \Z^d$, it suffices to show $S \cong S^d_1 - (s_d - 1)e_d$. Let $\varphi : \R^d \rightarrow \R^d$ be the linear mapping
given by $\varphi : e_i \mapsto e_i + \frac{s_d - 1}{s_i}e_d$ for $i \in \{1, \ldots, d-1\}$ and $\varphi : e_d \mapsto e_d$. Obviously,
$\varphi$ maps $S^d_1 - (s_d - 1)e_d$ to $S$. Furthermore, $\varphi$ preserves integrality as $s_i$ divides $s_d - 1$ for every $i \in \{1, \ldots, d-1\}$.
The associated transformation matrix is thus an integral lower triangular matrix and all entries on the diagonal are $1$ and hence, $\varphi$ is a unimodular
transformation.
\end{proof}

{\em Assertion (b) for $l=1$.}
Assume $l=1$. As shown above, the simplex $S$ with $S \cong S^d_1$ fulfils \eqref{eq:face_vol_bound} 
with equality for $l=1$. Conversely, let $l=1$ and let $S \in \Sd{1}$ be an arbitrary simplex satisfying \eqref{eq:face_vol_bound} with equality.
Let $\beta_1 \ge \ldots \ge \beta_{d+1} > 0$
be the barycentric coordinates of the interior integral point of $S$ and let $p_1, \ldots, p_{d+1}$ the corresponding vertices of $S$.
By our assumptions and Theorems~\ref{thm_faces_bc} and \ref{thm_single_low}, we have
\begin{align}
 2(s_d - 1) = \max_{1 \le i < j \le d+1}\vol_{\Z}([p_i,p_j]) \le \max_{1 \le i < d+1}\frac{1}{\beta_i} = \frac{1}{\beta_d} \le 2(s_d-1). \label{eq:l1_one}
\end{align}
This implies $\beta_d = 1/(2(s_d - 1))$. By Theorem~\ref{main_bc_bound}(b), 
\begin{align}
(\beta_1, \ldots, \beta_{d+1}) = \left(\frac{1}{s_1}, \ldots, \frac{1}{s_{d-1}}, \frac{1}{2(s_d - 1)}, \frac{1}{2(s_d - 1)} \right). \label{eq:opt_config}
\end{align}
Thus, by Theorem~\ref{thm_faces_bc} and \eqref{eq:opt_config}, $[p_d, p_{d+1}]$ is the unique edge having normalized volume $2(s_d - 1)$.

As the barycentric coordinates of the interior integral point of $S$
fulfil \eqref{eq:opt_config}, we can apply Lemma~\ref{lem_hyperplane_intersection} and obtain
$S \cong \conv(T^{d-1}_{1,d} \times \{0\} \cup \{\pm(a, h)\})$ for some $h \in \N$ and $a \in \{0, \ldots, h-1\}^{d-1}$.
Thus, to prove $S \cong S_d^1$ it suffices to show that $h = s_d - 1$ and $a = o$. Because the barycentric coordinates associated
with $\pm(a,h)$ are $\beta_d$ and $\beta_{d+1}$, we have that $[p_d, p_{d+1}]$
is the edge with endpoints $\pm(a,h)$.
Therefore, $2(s_d - 1) = \vol_{\Z}([p_d,p_{d+1}]) = \vol_{\Z}([(a,h),-(a,h)]) = \gcd(2(a,h))$.
This yields $h \ge s_d - 1$. On the other hand, $h > s_d - 1$ yields $\vol(S) > \frac{2(s_d - 1)}{d}\vol(T^{d-1}_{1,d}) = \frac{2(s_d - 1)^2}{d!}$.
Since $S \in \Sd{1}$, this is
a contradiction to Inequality~\eqref{eq:face_vol_bound}. Hence, $2(s_d - 1) = \gcd(2(a,s_d - 1))$.
As all components of $a$ are between $0$ and $h-1$, this also implies $a = o$, because otherwise $\gcd(2(a,s_d - 1)) < 2(s_d - 1)$. 
By Claim~\ref{clm_equiv_of_s1d}, this implies $S \cong S^d_1$.

{\em Assertion (b) for $l=d$.}
Assume $l=d$. By Theorem~\ref{thm_faces_bc} and Theorem~\ref{thm_location}(d), we have
\begin{align}
 \vol(S) \le \frac{1}{d!\beta_1 \cdots \beta_d} \le \frac{2}{d!} (s_d - 1)^2. \label{eq:all_ineq2}
\end{align}
We assume that $S$ fulfils \eqref{eq:face_vol_bound} with equality for $l=d$, i.e. $\vol(S) = \frac{2}{d!}(s_d - 1)^2$. It follows that
both inequalities in \eqref{eq:all_ineq2} are fulfilled with equality. Hence by Theorem~\ref{thm_location}(e) and because $d \ge 4$, we have
that the barycentric coordinates of the interior integral point of $S$ fulfil \eqref{eq:opt_config}. Therefore, we can again apply
Lemma~\ref{lem_hyperplane_intersection} to obtain $S \cong \conv(T \times \{0\} \cup \{\pm(a, h),\})$, where $T := T^{d-1}_{1,d}$, $h \in \N$ and 
$a \in \{0, \ldots, h-1\}^{d-1}$. Thus, 
\[
 \vol(S) = \frac{2h \vol(T)}{d} = \frac{2h}{d!} (s_d - 1).
\]
On the other hand, by assumption, we have $\vol(S) = \frac{2}{d!}(s_d - 1)^2$. This yields $h = s_d - 1$. It remains to show that 
$a = o$. By applying an appropriate unimodular transformation if necessary, we may assume $S = \conv(T \times \{0\} \cup \{\pm(a, s_d - 1)\})$.
We consider $S' \subseteq \R^{d-1}$ such that $S' \times \{1\}$ is the hyperplane section $S' \times \{1\} = S \cap (\R^{d-1} \times \{1\})$ of $S$. Then
\[
S' = \frac{h-1}{h} T + \frac{1}{h} a = \conv \left(\left\{o\right\} \cup \setcond{\frac{s_d-2}{s_d-1}s_i e_i}{i \in \{1, \ldots, d-1\}}\right) + \frac{1}{h}a.
\]
Observe that $\frac{h-1}{h}T$ fulfils the assumptions of Lemma~\ref{lem_cov_min_simplex} because
\[
 \frac{s_d - 1}{s_d - 2} \sum_{i=1}^{d-1}\frac{1}{s_i} = \frac{s_d - 1}{s_d - 2} \left(1 - \frac{1}{s_d - 1} \right) = 1.
\]
Thus, $\frac{h-1}{h} T + \frac{1}{h} a$ is lattice-free if and only
if $a/h \in \Z^{d-1}$. Assume the contrary, i.e. $a/h \not\in \Z^{d-1}$. Then 
there exists a point $p \in \intr{S'} \cap \Z^{d-1}$ and hence also in $\relintr{S'} \cap \Z^d$. Note that we also have
$(1, \ldots, 1, 0) \in \relintr{T \times \{0\}}$ and thus $(1, \ldots, 1,0) \in \intr{S} \cap \Z^d$. In other words, 
$S$ contains more than one integral point in its interior, a contradiction. Hence, $a/h \in \Z^{d-1}$ and since
$a \in \{0, \ldots, h-1\}^{d-1}$, this implies $a = o$. Thus, by Claim~\ref{clm_equiv_of_s1d},
we have $S \cong S_1^d$.
\end{proof}

\begin{proof}[Proof of Corollary~\ref{applying_blichfeldt}]
 This follows immediately from Theorem~\ref{main_volume} and Blichfeldt's Theorem~\ref{thm_blichfeldt}.
\end{proof}

\begin{remark}\thmtitle{On minimal $h$-faces of simplices in $\Sd{1}$}
 Using the results from Section~\ref{sec:auxiliary}, it is also possible to give upper bounds on the volume of minimal $h$-faces.
More precisely, we introduce the following functions on $\Sd{1}$: 
Let $d \ge 4$ and $g,h \in \{1, \ldots, d-1\}$ with $g < h$. Let
\[
\nu_h(S) := \min_{F \in \mathcal{F}_h(S)} \vol_{\Z}(F)  \qquad \forall \; S \in \Sd{1}
\] 
and let
\[
\gamma_{h,g}(S) := \min_{H \in \cF_h(S)} \max_{G \in \cF_g(H)} \vol_{\Z}(G) \qquad \forall \; S \in \Sd{1}.
\]
By \eqref{eq:face_vs_bc_max}, bounding $\nu_h(S)$ from above corresponds to bounding $\beta_1 \cdots \beta_h$ from below, which can be done using
Theorem~\ref{thm_location}(b). Furthermore, in view of \eqref{eq:face_vs_bc_max} and \eqref{eq:face_vs_bc_min}, one can show that bounding
$\gamma_{h,g}(S)$ corresponds to bounding $\beta_{h-g+1} \cdots \beta_h$ from below. Again, one can apply Theorem~\ref{thm_location}(b) to obtain a bound.
\end{remark}

\begin{remark}\thmtitle{On dimensions $d \le 3$}
At this point, we want to give a short overview about the validity of the results proved above in dimensions one, two and three. While all results hold for $d \ge 4$, 
Theorem~\ref{main_bc_bound} holds even for arbitrary dimension $d$. This is not true for 
the statements of Theorem~\ref{main_volume}.
The assumption of $d \ge 3$ in Theorem~\ref{main_volume} cannot be relaxed:
it fails to hold for $d=2$, since the triangle $T^2_{1,1} = \conv(\{o, 3e_1, 3e_2\})$ is the volume maximizer in $\mathcal{P}^2(1)$ (see \cite{MR1039134})
and has volume $4.5$, 
while the upper bound in the theorem would yield $4$. 
Regarding Theorem~\ref{main_volume}(b), we have to differ between the cases $l=1$ and $l=d$. The assumption of $d \ge 4$ cannot be relaxed for $l = d$, as
we know from the enumeration contained in \cite{MR2760660} that both $S^3_1$ and the 
tetrahedron $T^3_{1,2} = \conv(\{o, 2e_1, 6e_2, 6e_3\})$ satisfy the inequalities in Theorem~\ref{main_volume}(a) with equality for $l=d$ and $l=d-1$.
For $l=1$, however, we need to have minimal value for $\beta_d$, which is the case if and only if the tuple
of barycentric coordinates satisfies \eqref{eq:opt_config}; see Theorem~\ref{thm_single_low},
regardless of the value of $d$. Hence, for $l=1$ and $d \ge 1$, $\max_{F \in \cF_1(S)} \vol_{\Z} (F)= 2(s_d - 1)$ holds if $S \cong S^d_1$. 
For $l=1$ and $d \in \{1,2\}$, the reverse implication is also true, i.e., $\max_{F \in \cF_1(S)} \vol_{\Z} (F) = 2(s_d - 1)$ holds if and only if $S \cong S^d_1$
(as can be seen from the enumeration in \cite{MR1039134}).
Corollary~\ref{applying_blichfeldt} is also valid for $d \in \{1,2\}$, where the case $d=1$ is trivial and the case $d=2$
can be seen again from \cite{MR1039134}.
\end{remark}

\section{Proofs of results on $\Sd{1}$ involving dualization \\ (Theorems~\ref{mahler} and \ref{thm_dual_face_volumes})}\label{sec:dual_proofs}

We turn our attention to duals of simplices in $\Sd{1}$.
The following is a basic result from standard duality (see also \rescite{ben}{Proposition 3.6}).

\begin{proposition}\label{dual_same_bc}
Let $S$ be a $d$-dimensional simplex such that $o \in \intr{S}$ and let $v_1, \ldots, v_{d+1}$ be the vertices of $S$. Let
$\beta_1, \ldots, \beta_{d+1}$ be the barycentric coordinates of $o$ such that $o = \sum_{i=1}^{d+1} \beta_i v_i$. 
Then, there exists a unique sequence $u_1, \ldots, u_{d+1} \in \R^d$ such that 
\begin{align}
 \sprod{u_i}{v_j} = 1 \quad \forall \; i,j \in \{1, \ldots, d+1\}, \;  i \neq j. \label{eq:new_dual_prop_sprod}
\end{align}
For this sequence, one has:
\begin{empheq}{align}
  \vertset{S^{\ast}} &= \{u_1, \ldots, u_{d+1}\}, &\; \label{eq:new_dual_prop_vert}\\
  o &=\beta_1 u_1 + \ldots + \beta_{d+1} u_{d+1}, &\; \label{eq:neq_dual_prop_bc}\\
  \sprod{u_i}{v_i} &= 1 - \frac{1}{\beta_i} \quad \forall \; i \in \{1, \ldots, d+1\}. \label{eq:new_dual_prop_diag}
\end{empheq}
\end{proposition}

Next, we prove Theorem~\ref{mahler}. For that purpose, we need the following proposition.

\begin{proposition}\label{mahler_bary}
 Let $d \ge 1$ and $S \subseteq \R^d$ be a $d$-dimensional simplex such that $o \in \intr{S}$.
Let $\beta_1, \ldots, \beta_{d+1}$ be the barycentric coordinates
of $o$ with respect to $S$. Then
\[
 \vol(S) \vol(S^{\ast}) = \frac{1}{(d!)^2 \beta_1 \cdots \beta_{d+1}}.
\]
\end{proposition}

\begin{proof}[Proof sketch.]
 This result can be deduced from \cite[Proposition~3.6]{ben}. Since the language used in that paper differs from the one used here, we give a proof sketch for
Proposition~\ref{mahler_bary} as a service to the reader.

Let $v_1, \ldots, v_{d+1}$ be the vertices of $S$ such that $o = \sum_{i=1}^{d+1} \beta_i v_i$ and let $u_1, \ldots, u_{d+1}$ as in \eqref{eq:new_dual_prop_sprod}.
We introduce the matrices
\[
 V = \begin{pmatrix}
      v_1 & \cdots & v_{d+1}\\ 1 & \cdots & 1
     \end{pmatrix}
\quad \text{and} \quad
 U = \begin{pmatrix}
      u_1 & \cdots & u_{d+1}\\ -1 & \cdots & -1
     \end{pmatrix}.
\]
Then, one has $\vol(S) = \frac{1}{d!}|\det V|$ and $\vol(S^{\ast}) = \frac{1}{d!}|\det U|$. Furthermore, in view of \eqref{eq:new_dual_prop_sprod} and
\eqref{eq:new_dual_prop_diag}, we have that $V^{\top} U$ is a diagonal matrix with diagonal elements $-1/\beta_1, \ldots, -1/\beta_{d+1}$.
We conclude that $(d!)^2 \vol(S) \vol(S^{\ast}) = |\det V^{\top} U| = \frac{1}{\beta_1 \cdots \beta_{d+1}}$.
\end{proof}

\begin{proof}[Proof of Theorem~\ref{mahler}.]
By Proposition~\ref{mahler_bary} 
and by applying Lemma~\ref{i_and_k}(d) with $n = d+1$ and $\alpha_i = 1$ for every $i \in \{1, \ldots, n\}$,
we get the desired upper bound. Since, by Lemma~\ref{i_and_k}(d) and Theorem~\ref{main_bc_bound}(b), the product $\beta_1 \cdots \beta_{d+1}$ is minimal if and only if $\beta_{d+1}$ is minimal,
Theorem~\ref{main_bc_bound}(c) yields that the upper bound is attained if and only if $S = T^d_{1,d+1}$. 
The lower bound, meanwhile, follows from the fact that $\sum_{i=1}^{d+1} \beta_i = 1$ and hence by the inequality for the geometric and arithmetic means,
\[
 \beta_1 \cdots \beta_{d+1} \le \frac{1}{(d+1)^{d+1}}
\]
with equality if and only if $\beta_1 = \ldots = \beta_{d+1} = 1/(d+1)$. The latter is the case if and only if the unique interior integral point of $S$ is its centroid.
\end{proof}


\begin{proposition}\label{dual_one_ilp}
 Let $P$ be an integral $d$-dimensional polytope such that $o \in \intr{P}$. Then $\intr{P^{\ast}} \cap \Z^d = \{o\}$.
\end{proposition}

Based on the previous propositions, we can now show the following lemma regarding the dual of $S^d_1$. This will then be used to characterize the
cases in which the dual of a simplex $S \in \Sd{1}$ has maximal volume or an edge of maximal length, respectively. We show that in this
case, $S^{\ast}$ is equivalent to $S^d_1$ and hence, we can use the following lemma to describe $S$.

\begin{lemma}\label{lem_dual_to_s1d}
 Let $S = \conv((T^{d-1}_{1,d} \times \{0\}) \cup \{\pm e_d\}) - (1, \ldots, 1,0)$. Then $S^{\ast} \cong S^d_1$.
\end{lemma}

\begin{proof}
For convenience, let $\tilde{e} := e_1 + \ldots + e_{d-1}$.
 Let us first observe that the barycentric coordinates $\beta_1 \ge \ldots \ge \beta_{d+1}$ of $o$ with respect to $S$ satisfy 
\[
  (\beta_1, \ldots, \beta_{d+1}) := \left(\frac{1}{s_1}, \ldots, \frac{1}{s_{d-1}}, \frac{1}{2(s_d - 1)}, \frac{1}{2(s_d - 1)}\right).
\]
Furthermore, Proposition~\ref{dual_one_ilp} and \eqref{eq:neq_dual_prop_bc} yield that $o$ is the unique interior integral point of $S^{\ast}$ and 
$o$ has barycentric coordinates $\beta_1, \ldots, \beta_{d+1}$ with respect to $S^{\ast}$. Now define the facets $F_i$ of $S$
as $F_i := \conv(\vertset{S} \setminus \{s_i e_i - \tilde{e}\})$ for $i \in \{1, \ldots, d-1\}$. 
We set $u_i = -e_i$ for $i \in \{1, \ldots, d-1\}$ and
\[
 u_d = (s_d - 1)e_d + \sum_{j=1}^{d-1} \frac{s_d - 1}{s_j} e_j, \qquad u_{d+1} = - (s_d - 1)e_d + \sum_{j=1}^{d-1} \frac{s_d - 1}{s_j} e_j 
\]
and claim that $\vertset{S^{\ast}} = \{u_1, \ldots, u_{d+1}\}$.
Observe
that for $i,j \in \{1, \ldots, d-1\}$ with $j \neq i$ we have $\sprod{s_j e_j - \tilde{e}}{-e_i} = \sprod{-\tilde{e}}{-e_i} = 1$ 
as well as $\sprod{\pm e_d - \tilde{e}}{-e_i} = 1$.
Furthermore, for $i \in \{1, \ldots, d-1\}$, we have
\[
 \sprod{s_i e_i - \tilde{e}}{u_d} = - \sum_{k=1}^{d-1} \frac{s_d - 1}{s_k} + s_i \frac{s_d - 1}{s_i} = \left(- 1 + \frac{1}{s_d - 1} + 1\right)(s_d - 1) = 1 
\]
and analogously, $\sprod{e_d - \tilde{e}}{u_d} = 1$. By an analogous computation, one can also verify $\sprod{s_i e_i - \tilde{e}}{u_{d+1}} = \sprod{- e_d - \tilde{e}}{u_{d+1}} = 1$.
In view of \eqref{eq:new_dual_prop_sprod} and \eqref{eq:new_dual_prop_vert}, this suffices to show that $\vertset{S^{\ast}} = \{u_1, \ldots, u_{d+1}\}$.
To complete the proof, we have to show that the simplex
\[
 S^{\ast} = \conv \left( \left\{ e_1, \ldots, e_{d-1}, \pm(s_d - 1)e_d - \sum_{j=1}^{d-1} \frac{s_d - 1}{s_j}e_j  \right\} \right)
\]
is unimodularly equivalent to $S^d_1$. To this effect, we first use an integral translation that moves $-\sum_{j=1}^{d-1} \frac{s_d - 1}{s_j}e_j$
into the origin. The resulting simplex is
\[
 Q_1 := \conv \left( \left\{ e_1 + \sum_{j=1}^{d-1} \frac{s_d - 1}{s_j}e_j, \ldots, e_{d-1}+ \sum_{j=1}^{d-1} \frac{s_d - 1}{s_j}e_j, \pm(s_d - 1)e_d \right\} \right)
\]
and we want to show that $Q_1$ is unimodularly equivalent to the simplex
\[
 Q_2 := \conv( \left\{ s_1 e_1, \ldots, s_{d-1} e_{d-1}, \pm(s_d - 1)e_d \right\}).
\]
Both of these simplices have $\pm(s_d - 1)e_d$ as vertices. Hence, we can turn our attention to the linear mapping 
$\varphi: \R^{d-1} \times \{0\} \rightarrow \R^{d-1} \times \{0\}$ given by
\[
\varphi: \; e_i \mapsto \frac{1}{s_i}\left(e_i + \sum_{j=1}^{d-1}\frac{s_d - 1}{s_j} e_j\right) \qquad \forall \; i \in \{1, \ldots, d-1\},
\]
which maps the remaining vertices of $Q_2$ onto those of $Q_1$.
Once we have shown $\varphi(\Z^{d-1} \times \{0\}) = \Z^{d-1} \times \{0\}$, we are done.
For $i \in \{1, \ldots, d-1\}$, the integrality of the coefficients can be seen quickly from the fact that $s_i$ divides $s_d - 1$ for $i \in \{1, \ldots, d-1\}$, hence $\frac{s_d - 1}{s_j s_i}$ is obviously
an integer for every $j \in \{1, \ldots, j\}$ with $j \neq i$. For the coefficient of $e_i$, we have
\[
 \frac{1}{s_i} + \frac{s_d - 1}{(s_i)^2} = \frac{1 + s_1 \cdots s_{i-1}s_{i+1} \cdots s_{d-1}}{s_i} = \frac{1 + (s_i - 1)s_{i+1} \cdots s_{d-1}}{s_i}
\]
is also an integer, as $s_i$ divides $s_k - 1$ for $k \in \{i+1,\ldots,d-1\}$ and hence the numerator of this fraction is $0 \modulo{s_i}$.
We claim that $\varphi^{-1}: \R^{d-1} \times \{0\} \rightarrow \R^{d-1} \times \{0\}$ is the linear mapping $\psi$ given by $\psi: \; e_i \mapsto s_i e_i - \tilde{e}$ for all
$i \in \{1, \ldots, d-1\}$. Indeed,
\begin{align*}
 \psi (\varphi(e_i)) & = \frac{1}{s_i}\left(s_i e_i - \tilde{e} +\sum_{j=1}^{d-1}\frac{s_d - 1}{s_j} (s_j e_j - \tilde{e}) \right)\\
& = e_i + \frac{1}{s_i}\left(-\tilde{e} + (s_d - 1)\tilde{e} - \tilde{e}(s_d - 1)\left(1 - \frac{1}{s_d - 1}\right)\right)\\
& = e_i
\end{align*}
in view of Proposition~\ref{lem_syl_seq}(b). Obviously, $\psi$ preserves integrality. Thus, we get the desired statement.
\end{proof}

\begin{proof}[Proof of Theorem~\ref{thm_dual_face_volumes}]
The proof is again divided into four parts: the proofs of \eqref{eq:dual_vol} and assertion (a) are followed by proofs
of assertion (b) for $l=d$ and $l=1$, respectively.

{\em Inequality \eqref{eq:dual_vol}.}
Let $l \in \{1, \ldots, d\}$.
Let $\beta_1 \ge \ldots \ge \beta_{d+1}$ denote the barycentric coordinates of $o$
with respect to $S$. Then \eqref{eq:neq_dual_prop_bc} yields that $o$ has the same
barycentric coordinates with respect to $S^{\ast}$.
Proposition~\ref{dual_one_ilp} guarantees that $S^{\ast}$ is a $d$-dimensional simplex with precisely one interior integral point.
Hence, by \eqref{eq:face_vs_bc_max}, we have that for every $l$-dimensional face $F$,
\begin{align}
 \vol_{\Z}(F) \le \frac{1}{l!}\prod_{i=d-l+1}^{d} \frac{1}{\beta_i} \le \frac{2(s_d-1)^2}{l!(s_{d-l+1}-1)}, \label{eq:dual_eq_case}
\end{align}
where the second inequality follows from Theorem~\ref{thm_location}(d). This proves \eqref{eq:dual_vol}.

{\em Assertion (a).} 
Let $S \cong \conv \left(T \times \{0\} \cup \{\pm e_d\}\right)$ and thus by Lemma~\ref{lem_dual_to_s1d} we have $S^{\ast} \cong S^d_1$. 
By Theorem~\ref{main_volume}(a), for every $l \in \{1, \ldots, d\}$, there
is a face $F \in \cF_l(S^{\ast})$ such that
\[
 \vol_{\Z}(F) = \frac{2(s_d-1)^2}{l!(s_{d-l+1}-1)}.
\]
Thus, $S$ fulfils \eqref{eq:dual_vol} with equality for $l \in \{1, \ldots, d\}$.

{\em Assertion (b) for $l=d$.}
Let $S \in \Sd{1}$ be such that equality holds in \eqref{eq:dual_vol} for $l=d$. 
Then, by Theorem~\ref{thm_location}(e),
\begin{align}
  (\beta_1, \ldots, \beta_{d+1}) = \left(\frac{1}{s_1}, \ldots, \frac{1}{s_{d-1}}, \frac{1}{2(s_d - 1)}, \frac{1}{2(s_d - 1)}\right). \label{eq:bc_tuple}
\end{align}
By Proposition~\ref{mahler_bary}, we have
\[
 \vol(S^{\ast}) = \frac{1}{(d!)^2 \beta_1 \cdots \beta_{d+1} \vol(S)}.
\]
From Lemma~\ref{lem_hyperplane_intersection}, we know that $S \cong \conv((T \times \{0\}) \cup \{\pm(a, h)\})$, where $h \in \N$
and $a \in \{0, \ldots, h-1\}^{d-1}$. We obtain
\begin{align*}
 \vol(S^{\ast}) = \frac{1}{(d!)^2 \beta_1 \cdots \beta_{d+1} \frac{2h}{d}\vol(T)} = \frac{1}{d!2h \beta_1 \ldots \beta_{d+1}(s_d-1)}.
\end{align*}
From $\beta_1 \cdots \beta_{d+1} = \frac{1}{4} (s_d - 1)^{-3}$, we get
\[
 \vol(S^{\ast}) = \frac{2 (s_d - 1)^2}{d!h}.
\]
Since $\vol(S^{\ast}) = \frac{2}{d!}(s_d - 1)^2$, we have $h=1$.

{\em Assertion (b) for $l=1$.}
Let $l=1$ and let $S \in \Sd{1}$ be such that equality holds in \eqref{eq:dual_vol}. 
Applying Theorem~\ref{thm_faces_bc} for the case of one-dimensional faces and Theorem~\ref{thm_single_low}, we deduce 
that \eqref{eq:bc_tuple}
are again the barycentric coordinates of $o$ with respect to both $S$ and $S^{\ast}$.
Applying Lemma~\ref{lem_hyperplane_intersection}, we get $S \cong \conv\left(T \times \{o\} \cup \{\pm(a,h)\}\right)$, 
where $h \in \N$ and $a \in \{0, \ldots, h-1\}^{d-1}$. Note that $(a,h)$ and $-(a,h)$ correspond to the barycentric coordinates 
$\beta_{d}$ and $\beta_{d+1}$, respectively.
We want to show that $h = 1$. We write $\tilde{e} := e_1 + \ldots + e_{d-1} \in \R^d$.
By applying a unimodular transformation if necessary, we may assume $S = \conv (T \times \{0\} \cup \{\pm(a,h)\}) - \tilde{e}$.
By \eqref{eq:neq_dual_prop_bc}, we can write
\begin{align}
 o = \sum_{i=1}^{d+1} \beta_i u_i, \label{eq:dual_l1}
\end{align}
where $u_1, \ldots, u_{d+1}$ denote the vertices of $S^{\ast}$.
By Theorem~\ref{thm_faces_bc} we have that for every edge $E := [u_i, u_j]$ of $S^{\ast}$, 
\[
\vol_{\Z}(E) \le \frac{1}{\min \{\beta_i,\beta_j\}}. 
\]
As all barycentric coordinates are known, this implies that the unique edge of $S^{\ast}$ with largest normalized volume is $[u_{d}, u_{d+1}]$ and since 
we assumed equality in Theorem~\ref{thm_dual_face_volumes},
we have $\vol_{\Z}([u_d, u_{d+1}]) = 2(s_d - 1)$. Hence, $u_{d} - u_{d+1} = 2(s_d-1) b$ for some integral unit vector $b$. 
We want to show $S^{\ast} \cap \lspan(\{u_{d} - u_{d+1}\}) = [b,-b]$. To see this, observe that
\[
 S^{\ast} \cap \lspan(\{u_{d} - u_{d+1}\}) = \left[ \sum_{i=1}^{d-1} \beta_i u_i + (\beta_{d} + \beta_{d+1}) u_d, \sum_{i=1}^{d-1} \beta_i u_i + (\beta_d + \beta_{d+1}) u_{d+1} \right].
\]
In view of \eqref{eq:dual_l1}, we have 
\[
 \sum_{i=1}^{d-1} \beta_i u_i + (\beta_{d} + \beta_{d+1}) u_d = \beta_{d+1} (u_d - u_{d+1}) = \frac{2(s_d - 1)}{2(s_d - 1)}  b = b.
\]
An analogous computation for $-b$ shows that $S^{\ast} \cap \lspan(\{u_{d} - u_{d+1}\}) = [b,-b]$.
This implies the equalities $\rho(S^{\ast},b) = 1$ and $\rho(S^{\ast},-b) = 1$ for the radius function of $S^{\ast}$.
Now we have the following relation between the radius function and the support function:
\[
 \rho(S^{\ast},b) h(S,b) = 1, \qquad \rho(S^{\ast},-b) h(S,-b) = 1;
\]
see, e.g., \rescite{MR1216521}{Remark 1.7.7}.
This yields $h(S,b) = h(S,-b) = 1$. We now show that $b \in \lspan(\{e_d\})$. In view of \eqref{eq:new_dual_prop_sprod},
for every $i \in \{1, \ldots, d-1\}$ and the vertex $s_i e_i - \tilde{e}$ of $S$ we have the equalities 
$\sprod{u_d}{s_i e_i - \tilde{e}} = 1$ and 
$\sprod{u_{d+1}}{s_i e_i - \tilde{e}}=1$. 
As a consequence, we have $\sprod{u_{d+1} - u_d}{s_i e_i - \tilde{e}} = 0$ for every $i \in \{1, \ldots, d-1\}$.
Since $\lspan(\setcond{s_i e_i - \tilde{e}}{i \in \{1, \ldots, d-1\}}) = \R^{d-1} \times \{0\}$, we obtain $u_{d+1} - u_d \in \lspan(\{e_d\})$ and hence $b \in \lspan(\{e_d\})$.
Together with $h(S,b) = h(S,-b) = 1$, this yields $S \subseteq \R^{d-1} \times [-1,1]$.
This shows $h = 1$ and
$S \cong \conv(T \times \{0\} \cup \{\pm e_d\})$. 
\end{proof}

In the remainder of this section we use dualization to show that a uniqueness results as in Theorem~\ref{main_bc_bound}(c) cannot be obtained for $i < d+1$.

\begin{proposition}\label{prop_duals_of_t}
 Let $i \in \{1, \ldots, d+1\}$ and $e := e_1 + \ldots + e_d$. Then the simplex $(T^d_{1,i} - e)^{\ast}$ is integral.
\end{proposition}

\begin{proof}
Let $i \in \{1, \ldots, d+1\}$ and write $Q^d_{1,i} := T^d_{1,i} - e$. For $j \in \{1, \ldots, d\}$, let $F_j$ denote the facet of $T^d_{1,i}$ not containing 
$c_j e_j$, where $c_j e_j$ is a vertex of $T^d_{1,i}$. In other words, $c_j = s_j$ for $j \in \{1, \ldots, i-1\}$
and $c_j = (d-i+2)(s_i - 1)$ for $j \in \{i, \ldots, d\}$. Then the facet $F_j$ yields $-e_j$ as a vertex of $(Q^d_{1,i})^{\ast}$. To see this, observe that
for any $c \in \R$,
$\sprod{c e_i - e}{-e_j} = 1$ for $i \neq j$, $i,j \in \{1, \ldots, d\}$.
Hence, $(Q^d_{1,i})^{\ast} = \conv(\{-e_1, \ldots, -e_d, v\})$, where $v$ is the vertex corresponding to the facet of $T^d_{1,i}$ which does not contain
$o$. Since the barycentric coordinates of $o$ with respect to $(Q^d_{1,i})^{\ast}$ are known, we can write
\[
 o = - \sum_{j=1}^{i-1}\frac{e_j}{s_j} - \sum_{j=i}^d \frac{e_j}{(d-i+2) (s_{i} - 1)} + \frac{v}{(d-i+2) (s_{i} - 1)}.
\]
Hence
\[
 v = \sum_{j=1}^{i-1} \frac{(d-i+2)(s_{i} - 1)}{s_j}e_j + \sum_{j=i}^d e_{j}.
\]
Since $s_j$ divides $s_{i} - 1$ for $j \in \{1, \ldots, i-1\}$, all vertices of $(Q^d_{1,i})^{\ast}$ are integral.
\end{proof}

\begin{remark}\thmtitle{Non-uniqueness of the minimizers of $\beta_i$}\label{rem_not_unique_bc}
The previous proposition implies that for $i \in \{1, \ldots, d+1\}$, the lattice simplices $T^d_{1,i} - e$ are reflexive in the sense of \cite{ben}.
As the following volume computation shows, $T^d_{1,i} - e$ and $(T^d_{1,i} - e)^{\ast}$ are not unimodularly equivalent unless $i=d+1$, 
despite the fact that the 
origin has the same barycentric coordinates with respect to their vertices.
In particular, this shows that the minimizers of the $i$-th
barycentric coordinate are not unique unless $i=d+1$. Let $i \in \{1, \ldots, d\}$ and let $\beta_1, \ldots, \beta_{d+1}$ denote the barycentric coordinates
of $o$. Observe that by Proposition~\ref{mahler_bary}, we have
\[
 \vol(T^d_{1,i} - e)\vol((T^d_{1,i} - e)^{\ast}) = \frac{1}{(d!)^2 \beta_1 \cdots \beta_{d+1}}.
\]
As 
\[
 \vol(T^d_{1,i} - e) = \frac{1}{d!} (d-i+2)^{d-i+1}(s_{i} - 1)^{d-i+1} \prod_{j=1}^{i-1} s_j = \frac{1}{d!} (s_{i}-1)^{d-i+2}(d-i+2)^{d-i+1}
\]
and
\[
 \frac{1}{\beta_1 \cdots \beta_{d+1}} = (s_{i} - 1)^{d-i+3}(d-i+2)^{d-i+2},
\]
we have
\[
 \vol((T^d_{1,i} - e)^{\ast}) = \frac{1}{d!}(s_{i} - 1) (d-i+2).
\]
Hence, $\vol(T^d_{1,i} - e) \neq \vol((T^d_{1,i} - e)^{\ast})$ for $i \in \{1, \ldots, d\}$.
\end{remark}

\section{Proofs of results about the coefficient of asymmetry \\ (Theorems~\ref{thm_ca} and \ref{thm_vol_by_ca})}\label{ca_proofs}

We first reformulate the definition of the coefficient of asymmetry. This alternative definition then simplifies the proof of Theorem~\ref{thm_ca}.

\begin{proposition}\label{prop_alt_ca}\thmtitle{Alternative definition of $\ca(P,o)$}
 Let $P \subseteq \R^d$ be a polytope such that $o \in \intr{P}$. Then
\begin{align*}
 \ca(P,o) &= \min \setcond{\alpha \ge 0}{P \subseteq -\alpha P}\\
	  &= \min \setcond{\alpha \ge 0}{\vertset{P} \subseteq -\alpha P}.
\end{align*}
Thus, $\ca(P,o)$ is the unique $\alpha > 0$ such that $P \subseteq -\alpha P$ and $\vertset{P} \cap -\alpha \bd{P} \neq \emptyset$.
\end{proposition}

\begin{proof}
It is easy to see that we can rewrite $\ca(P,o)$ as
\begin{align*}
 \ca(P,o) & = \max\setcond{\frac{\rho(P,u)}{\rho(P,-u)}}{u \in \R^d\setminus \{o\}} \\
  & = \min\setcond{\alpha \ge 0}{\frac{\rho(P,u)}{\rho(P,-u)} \le \alpha \; \forall \; u \in \R^d\setminus \{o\}} \\
  & = \min\setcond{\alpha \ge 0}{\rho(P,u) \le \alpha \rho(P,-u) \; \forall \; u \in \R^d\setminus \{o\}}.
\end{align*}
Because $\alpha \rho(P,-u) = \rho(\alpha P, -u) = \rho(-\alpha P, u)$ and the fact that two polytopes $Q,Q'$ satisfy $\rho(Q,u) \le \rho(Q',u)$ 
for every $u \in \R^d\setminus \{o\}$ if and only if $Q \subseteq Q'$, we get
\begin{align}
 \ca(P,o) = \min\setcond{\alpha \ge 0}{P \subseteq -\alpha P} = \min \setcond{\alpha \ge 0}{\vertset{P} \subseteq -\alpha P}, \label{eq:ca_proof}
\end{align}
where the second equality in \eqref{eq:ca_proof} is obvious from the convexity of $P$.
This implies $\vertset{P} \cap -\ca(P,o)\bd{P} \neq \emptyset$, because otherwise $\bd{-\ca(P,o) P} \cap P = \emptyset$ 
and one could choose
$\alpha' < \ca(P,o)$ such that $P \subseteq -\alpha' P \varsubsetneq -\ca(P,o)P$, a contradiction. 
\end{proof}

\begin{proof}[Proof of Theorem~\ref{thm_ca}.]
By Proposition~\ref{prop_alt_ca}, there is a vertex of $P$ such that $\ca(P,o)$ is attained along the line going through this vertex and $o$.
Hence, we can choose $v \in \vertset{P}$ and $u \in \bd{P}$ such that $o \in [v,u]$ and $\ca(P,o) = \frac{\|v\|}{\|u\|}$.
Considering any proper face of $P$ that contains $u$ and using Caratheodory's theorem (see, e.g., \rescite{MR1940576}{I. 2.3}), we can find an integral simplex $T$ of dimension at most $d-1$ in $\bd{P}$ such that
$u \in \relintr{T}$. Since we have $o \in \intr{P}$, $o \in \relintr{S}$ and $S \subseteq P$, we have $\relintr{S} \subseteq \intr{P}$. Hence, $o$ is the only relative interior integral point of the simplex $S := \conv(T \cup \{v\})$. We denote the barycentric coordinate of $o$
(with respect to this simplex) associated with $v$ by $\beta$. Clearly,
$\|v\|/\|u\| = (1 - \beta) /\beta$. Let $h$ be the dimension of $S$. Then, in view of \eqref{eq:bc_bound_formula}, we have
\begin{align}
\ca(P,o) = \frac{1 - \beta}{\beta} \le s_{h+1} - 2 \le s_{d+1} - 2. \label{eq:ca_eq_case}
\end{align}
It remains to characterize the equality case. First, observe that if $P \cong T^d_{1,d+1}$, we have $\ca(P,o) = s_{d+1} - 2$ due to 
\eqref{eq:proving_hensley}. Now, let $P \in \Pd{1}$ be such that $\intr{P} \cap \Z^d = \{o\}$
and $\ca(P,o) = s_{d+1}-2$. The assumption $\ca(P,o) = s_{d+1}-2$ 
implies that for the simplex $S$, we have equality in both inequalities in \eqref{eq:ca_eq_case}. Hence, $S$
has dimension $d$ and furthermore, $o$ has the barycentric coordinate $1/(s_{d+1}-1)$ with respect to the vertex $v$ of $S$. Theorem~\ref{main_bc_bound}(c) yields that $S$ is unimodularly equivalent to $T^d_{1,d+1}$
and, in particular, the unimodular transformation mapping $S$ onto $T^d_{1,d+1}$ has to map $v$ onto $o$. Thus, $u$ is in
the relative interior of the facet of $S$ which is opposite to $v$. As $u$ is also in $\bd{P}$, we have that this facet is
in the boundary of $P$. We finish the proof by showing $P = S$. Assume that there exists some $y \in P \setminus S$. Then for some facet $F$ of $S$, one has
$\relintr{F} \subseteq \intr{P}$. By our previous arguments, this facet cannot be the one opposite to
$v$. As the unimodular transformation mapping $S$ onto $T^d_{1,d+1}$ maps $v$ onto $o$,
$F$ therefore has to be unimodularly equivalent to a facet of $T^d_{1,d+1}$ which contains $o$. It is easy
to check that each such facet of $T^d_{1,d+1}$ contains at least one integral point in its relative interior\footnote{Every simplex of the form
$\conv(\{o, a_1 e_1, \ldots, a_d e_d\})$ with $a_1, \ldots, a_d > 0$ and $1/a_1 + \ldots + 1/a_d < 1$ contains the point $\allone - e_i$ in the relative
interior of the facet opposite to $a_i e_i$ for each $i \in \{1, \ldots, d\}$. In view of Proposition~\ref{lem_syl_seq}~(b), the simplex $T^d_{1,d+1}$
is of that form.} and hence $\relintr{F}$
contains a point of $\Z^d \setminus \{o\}$. Since $\relintr{F} \subseteq \intr{P}$, this is a contradiction to $\intr{P} \cap \Z^d = \{o\}$.
\end{proof}

\begin{remark}
Using the argument of Izhboldin and Kurliandchik as outlined in Remark~\ref{ik_remark}, it is possible to
determine the maximizer of the asymmetry coefficient within $\Sd{1}$
without any use of Soundararajan's arguments.
\end{remark}

\begin{proof}[Proof of Theorem~\ref{thm_vol_by_ca}.]
As $\intr{P} \cap \Z^d = \{o\}$, Mahler's theorem (Theorem~\ref{thm_mahler}) asserts that $\vol(P) \le (1 + \ca(P,o))^d$. Applying
Theorem~\ref{thm_ca} completes the proof.
\end{proof}

In the previous proof, we made use of Mahler's theorem (see Theorem~\ref{thm_mahler}). It should be noted that Sawyer proved a slightly sharper
inequality than Mahler's in \cite{MR0061139} which does, however, not lead to asymptotically better results.

\begin{remark}\thmtitle{Asymmetry of the volume-maximizers in $\Pd{1}$}
 We ask whether
the simplex $S^d_1$ has maximal volume among all elements of $\Pd{1}$. The main reason for this question is the apparent link 
between asymmetry and large volume. 
Assume that $P$ has maximal volume among all elements of $\Pd{1}$. Then $\vol(P) \ge \vol(S^d_1) = \frac{2}{d!}(s_d - 1)^2$.
Let, without loss of generality, $o \in \intr{P}$. 
Then by Theorem~\ref{thm_mahler}, we have $\vol(P) \le (1 + \ca(P,o))^d$. This yields the lower bound
\[
\ca(P,o) \ge \left( \frac{2}{d!} (s_d-1)^2 \right)^{1/d} -1 \ge  2^{2^{d+o(d)}}.
\]
Thus, every volume maximizer in $\Pd{1}$ has asymmetry coefficient with respect to its interior integral point of double exponential order in $d$. 
As simplices are typical examples of highly `non-centrally-symmetric' polytopes, this leads to the question whether every volume maximizer must be a 
simplex or at least must be close to a simplex with respect to some metric.
A similar argumentation can be applied to $\Pd{k}$ with $k \ge 2$ using a generalization of Theorem~\ref{thm_mahler} by 
Lagarias and Ziegler \rescite{MR1138580}{Theorem 2.5} and Pikhurko's bound on the coefficient of asymmetry of elements in $\Pd{k}$ \rescite{MR1996360}{Theorem 4}.
Whether for $d \ge 3$ and $k \ge 1$, the simplex $S^d_k$ has maximal volume among all elements of $\Pd{k}$ is an open question.
However, for $k \ge 2$, it is not even known whether $S^d_k$ has maximal volume among all elements of $\Sd{k}$.
\end{remark}

\section{Proofs of results on the lattice diameter (Theorems~\ref{general_ld_bound} and \ref{lpf_ld_bound})}
\label{sec:diameter}

In this section, we want to review a part of \cite{MR2855866} which used the bounds on barycentric coordinates of the interior integral point
of a simplex $S \in \Sd{1}$ as they were given by Pikhurko in \cite{MR1996360}. In \cite{MR2855866}, as a byproduct on the way to proving finiteness of $\Pim$ (up to unimodular
transformation) in fixed 
dimension $d$, a bound on $\ld{P}$ was established for $P \in \Pim$ which depended
on Pikhurko's bounds. Since we now have the exact lower bound on the barycentric coordinates, it is worth to revisit the 
corresponding arguments of \cite{MR2855866}. This will allow to determine the sharp upper bound on the 
lattice diameter of elements of $\Plm$, which is a larger family than $\Pim$. We will repeat here only those details from \cite{MR2855866} necessary for our purpose but on the other hand aim to give a self-contained proof of Theorem
\ref{lpf_ld_bound}. We make
use of the following lemma:

\begin{lemma}\label{lem_capturing_simplex}\thmtitle{\rescite{MR2855866}{p. 6}}
 Let $P \subseteq \R^d$ be an integral polytope for which $\relintr{P}\cap \Z^d$ is not empty and let $a \in \relbd{P} \cap \Z^d$. 
Then $P$ contains an integral simplex $S$ of dimension $h$, for some $h \in \{1, \ldots, d\}$, such that $a \in \vertset{S}$ and
$|\relintr{S} \cap \Z^d| = 1$.
\end{lemma}

\begin{proof}[Proof of Theorems~\ref{general_ld_bound} and \ref{lpf_ld_bound}.]
We first prove the bounds on the lattice diameter given in the two theorems. Then, we characterize the equality cases.

{\em The inequalities in Theorems~\ref{general_ld_bound} and \ref{lpf_ld_bound}.}
We can assume $d \ge 2$. Let $P \subseteq \R^d$ be an integral polytope and let $m := \ld{P'}$, where $P'$ denotes the convex hull of $\intr{P} \cap \Z^d$.
There are two cases: $P' \neq \emptyset$, i.e. $m \ge 0$, or $P' = \emptyset$ and hence $m = -1$.
For most of the proof, our argumentation is the same for both cases, with the only difference that, if $P' = \emptyset$, we make the additional assumption
of maximality of $P$ (i.e. $P \in \Plm$) as in the formulation of Theorem~\ref{lpf_ld_bound}.
Choose a line $l$ such that $|P \cap \Z^d \cap l|$ is maximal, i.e., such that $\ld{P}$ is attained along $l$. 
If $\ld{P} \le m+2$, then the statements hold trivially since $s_d \ge 2$. Hence, we can assume $\ld{P} > m+2$. 
By applying some unimodular transformation,
we can assume $l = \lspan(\{e_d\})$. 
By $\pi: \R^d \rightarrow \R^{d-1}$ we denote the projection to the first $d-1$ components.
Then, the polytope $Q := \pi(P)$ is the projection of $P$ along $l$. 
Clearly, for every point $p \in P \cap \Z^d \cap l$, we have
$\pi(p) = o$. Since $\ld{P} > m+2$, it follows that $o \in \bd{Q}$. To see this, observe that otherwise $\relintr{l \cap P} \subseteq \intr{P}$.
Now $\pi^{-1}(o) \cap P$ contains $\ld{P} + 1 \ge m + 4$ integral points. This implies that $\intr{P}$ contains at least
$m+2$ collinear integral points, a contradiction to the choice of $m$. Note that $\intr{Q} \cap \Z^{d-1}$ is not empty.
In the case $\intr{P} \cap \Z^d \neq \emptyset$ (i.e. $m \ge 0$), this is obvious
as every interior point of $P$ has to be projected to a point in $\intr{Q}$. 
In the case $P' = \emptyset$, we made the additional assumption $P \in \Plm$. 
It is easy to see that under this assumption,
$\intr{\pi(P)} \cap \Z^{d-1}$ is not empty (see also \rescite{MR2855866}{Lemma 3.7}).

Hence, from Lemma~\ref{lem_capturing_simplex},
we know that for some $h \in \{1, \ldots, d-1\}$ there exists an $h$-dimensional simplex $S \subseteq Q$ which satisfies $o \in \vertset{S}$ and contains exactly one point of 
$\pi(\Z^d)$ in its relative interior.
Denote this point by $q$ and let $q := \sum_{i=0}^{h} \lambda_i v_i$ be a barycentric representation of $q$, where $v_i \in \vertset{S}$ for $i \in \{1, \ldots, h\}$ 
and $v_0 = o$.
For $x \in Q$, define $f(x)$ as the length of the intersection between $x+\lspan(\{e_d\})$ and $P$, or more formally, 
$f(x) := \len(\pi^{-1}(x) \cap P)$. Observe that $f(q) \le m+2$ as otherwise, there would be at least $m+2$ collinear points in the interior of $P$. 
Furthermore, by convexity of $P$, the function $f$ is concave on $Q$ and we can use 
Jensen's inequality (see, e.g., \rescite{MR1451876}{Theorem 4.3}) to obtain
\begin{align}
\ld{P'} + 2 & \ge m + 2 \ge f(q) = f \left( \sum_{i=0}^h \lambda_i v_i \right) \geq \sum_{i=0}^h \lambda_i f(v_i) \notag \\ 
& \geq \lambda_0 f(v_0) \geq \lambda_0 \ld{P} \label{eq:jensen}\\ 
&\geq \frac{1}{s_{h+1} - 1} \ld{P}  \geq \frac{1}{s_{d} - 1} \ld{P} \notag. 
\end{align}
The last line in \eqref{eq:jensen} comes from the fact that
$\lambda_0$ is a barycentric coordinate of the unique
interior integral point of an integral simplex of dimension at most $d-1$ and we can apply Theorem~\ref{main_bc_bound}.
This proves the inequalities in Theorems~\ref{general_ld_bound} and \ref{lpf_ld_bound}.

{\em Characterization of the equality cases.}
First, we show that in the case $P' = \emptyset$, $P \in \Plm$ as well as in the case $P' \neq \emptyset$, there exist polytopes for which the respective bounds are attained with equality. 
If $P' \neq \emptyset$ and hence $m \ge 0$, note that the simplex $S^d_{m+1}$
has $m+1$ interior integral points which are collinear (see \rescite{MR1138580}{Proposition 2.6}) and hence $\operatorname{ld}\left(\conv\left(\operatorname{int}\left(S^d_{m+1}\right) \cap \Z^d\right)\right) = m$. Furthermore, this simplex has
$[o, (m+2)(s_d - 1)e_d]$ as an edge. This edge contains $(m+2)(s_d - 1) + 1$ integral points 
and therefore, $\operatorname{ld}\left(S^d_{m+1}\right) \ge (s_d - 1) (m + 2)$.
Together with the bound in \eqref{eq:jensen}, this yields $\operatorname{ld}\left(S^d_{m+1}\right) = (s_d - 1) (m + 2)$. In the case $P' = \emptyset$, $P \in \Plm$, i.e. $m = -1$,
observe that the simplex $S^d_0$ does not contain an integral point in its interior but each facet of $S^d_0$ 
has an integral point in its relative interior. Thus, $S^d_0$ belongs to $\Pim$ and therefore also to $\Plm$. It is easy to check that
$\operatorname{ld}\left(S^d_0\right) \ge s_d - 1$ and by our previous arguments, this implies $\operatorname{ld}\left(S^d_0\right) = s_d - 1$.

Now, we assume that $P$ fulfils the inequality in Theorem~\ref{general_ld_bound} or \ref{lpf_ld_bound}, respectively, with equality.
More precisely, let $P \subseteq \R^d$ be a polytope with $\vertset{P} \subseteq \Z^d$ and $\ld{P'} = m$ for some $m \in \N \cup \{0, -1\}$
such that $\ld{P} = (m + 2)(s_d - 1)$, where $P' := \conv(\intr{P} \cap \Z^d)$. 
We proceed as above, i.e, we project along the direction in which the lattice diameter
is attained (we can again assume that this direction is $e_d$) and write $Q := \pi(P)$. 
Again, we can then construct the simplex $S$ as we did before.
Thus, $S$ is an integral simplex of dimension $h \le d-1$ with precisely one interior integral point and $v_0 := o$ as one of its vertices.
By our assumption on the lattice diameter of $P$, all inequalities in \eqref{eq:jensen} are fulfilled with equality for $P$.
The equality $s_{h+1} - 1 = s_d - 1$ then immediately implies $h = d-1$. Moreover, equality in 
\eqref{eq:jensen} also yields 
$\lambda_0 = 1/(s_d - 1)$. This implies
that the interior integral point of $S$ has a barycentric coordinate with minimal value. By Theorem~\ref{main_bc_bound}(b),
we know that $S \cong \conv(\{o, s_1 e_1, \ldots, e_{d-1} s_{d-1}\})$. 

Next, we want to show $Q = S$ by showing that all facets of $S$ are also facets of $Q$.
As $P$ fulfils all inequalities in \eqref{eq:jensen} with equality,
we have $f(v) = 0$ for all vertices $v \in \vertset{S} \setminus \{v_0\}$. 
This suffices to show $Q = S$ in the
case $d=2$ (and hence $\dim(Q) = \dim(S) = 1$), as the vertex of $S$ which is not $v_0$ has to be in $\bd{Q}$. Hence, we can assume $d \ge 3$.
Denote by $F$ the facet of $S$ such that $v_0 \not\in F$. Note that for all facets $G$ of $S$ with $G \neq F$, one has
$\relintr{G} \cap \Z^{d-1} \neq \emptyset$. 
If $G \not\subseteq \bd{Q}$, we have $\relintr{G} \subseteq \intr{Q}$ and one can apply 
Lemma~\ref{lem_capturing_simplex} and find a simplex $S' \subseteq G$ of dimension $h' \le d-2$ which has $v_0$ as a vertex
and contains an interior integral point $q'$. Let $v'_i$ be the vertices of $S'$ and let $q' = \sum_{i=0}^{h'} \lambda'_i v'_i$ with
$\sum_{i=0}^{h'} \lambda'_i = 1$ and $\lambda'_i > 0$ for all $i \in \{0, \ldots, h'\}$.
By the same construction as before, we get
\[
 m+2 \ge f(q') = f \left(\sum_{i=0}^{h'} \lambda'_i v'_i \right) \ge \lambda'_0 f(v'_0) \ge \lambda'_0 \ld{P}. 
\]
This yields $\lambda'_0 \le \frac{1}{s_d - 1}$ by our assumption on $\ld{P}$, a contradiction to \eqref{eq:bc_bound_formula} as $S'$
is of dimension at most $d-2$.
This shows that $G \subseteq \bd{Q}$ for all facets $G \neq F$ and therefore, $Q \subseteq \R^{d-1}_{\ge 0}$.
It remains to show $F \subseteq \bd{Q}$. Assume the contrary, then $\relintr{F} \subseteq \intr{Q}$ and therefore, $f(y) > 0$ for every $y \in \relintr{F}$.
Let $y'$ be the point of $\relintr{F}$ such that $q$ is in the relative interior of the line segment $[v_0,y']$. Again employing concavity of $f$, we have
\[
 m+2 \ge f(q) \ge \lambda_0 f(v_0) + (1 - \lambda_0) f(y') \ge m+2 + (1 - \lambda_0) f(y').
\]
As both $1 - \lambda_0$ and $f(y')$ are positive, this is a contradiction.

We have proven $Q = S$. In other words, we have shown that the projection of $P$ along $e_d$ is equivalent to 
$\conv(\{o, s_1 e_1, \ldots, e_{d-1} s_{d-1}\})$. 
We conclude the proof by constructing $P$ from this projection. First, we show that $P$ is a simplex in dimension $d$.
For each vertex $v_i$ of $Q$, with $i \in \{1,\ldots,d-1\}$, the point
$p_i := \pi^{-1}(v_i) \cap P$ is a vertex of $P$ (since $f(v_i) = 0$ as we have shown before). 
Furthermore, the endpoints of $\pi^{-1}(v_0) \cap P$ are two more vertices of $P$, 
denoted by $p_0$ and $p'_0$. 
Let $T := \conv(\setcond{p_i}{i \in \{0, \ldots, d-1\}} \cup \{p'_0\})$. By construction,
all facets of $T$ which contain both $p_0$ and $p'_0$ are in the boundary of $P$ (as their respective projections lie in the boundary of $Q$).
We can write
\[
 q_0 := \lambda_0 p_0 + \sum_{i=1}^{d-1} \lambda_i p_i, \qquad q_0' := \lambda_0 p_0' + \sum_{i=1}^{d-1} \lambda_i p_i,
\]
which yields a point in the respective relative interior of each of the remaining two facets of $T$. 
By construction, the segment $[q_0, q'_0]$ has length
\begin{align*}
\len([q_0,q'_0]) = \lambda_0 \len([p_0,p'_0]) = \lambda_0 (m+2) (s_d - 1) = m+2.
\end{align*}
Consequently, one has $q_0,q'_0 \in \bd{P}$ since otherwise, the segment $[q_0, q'_0]$ has $m+2$ integral points
contained in the interior of $P$, which contradicts $m = \ld{P'}$. Thus, all facets of $T$ are contained in the boundary of $P$ and hence $P = T$.

Finally, we show that the simplex $P$ is unimodularly equivalent to $S^d_{m+1}$.
Without loss of generality, we set $p_0 = o$ and $p'_0 = \ld{P} e_d =(m+2)(s_d - 1)e_d$. 
The points $p_i$, $i \in \{1, \ldots, d-1\}$, are of the form $p_i = s_i e_i + c_i e_d$
for some integers $c_i$. Note that there
exists some $c_0 \in \R$ such that $q_0 = (1, \ldots, 1, c_0)$
and $q_0' = (1, \ldots, 1, c_0 + m + 2)$. Moreover, $c_0 \in \Z$. To see this, assume the contrary, i.e. $c_0 \in \R \setminus \Z$. Then $[c_0, c_0 + m+2]$ contains
$m+2$ integral points which lie in the interior of $P$, a contradiction to the definition of $m$. From the definition of $q_0,q_0'$ we get
\[
 c_0 = \sum_{i=1}^{d-1} \lambda_i c_i. 
\]
By Proposition~\ref{div_property_sylv} and because $\lambda_i = 1/s_i$ for $i \in \{1, \ldots, d-1\}$, we know that $c_0 \in \Z$ implies that $s_i$ divides $c_i$ for every $i \in \{1, \ldots, d-1\}$.
Consider the linear mapping $\varphi$ defined by $e_i \mapsto e_i - \frac{c_i}{s_i}e_d$ for every $i \in \{1, \ldots, d-1\}$ and $e_d \mapsto e_d$. Because $s_i$ divides
$c_i$ for $i \in \{1, \ldots, d-1\}$, $\varphi(\Z^d) \subseteq \Z^d$. The inverse mapping to $\varphi$ is given by $e_i \mapsto e_i + \frac{c_i}{s_i}e_d$ for 
$i \in \{1, \ldots, d-1\}$ and $e_d \mapsto e_d$, which shows that $\varphi(\Z^d) = \Z^d$. Thus, $\varphi$ is a unimodular transformation.
Clearly, $\varphi(P) = S^d_{m+1}$. This proves that, up to unimodular equivalence, $P$ is a uniquely determined simplex, namely $P \cong S^d_{m+1}$.
\end{proof}

\begin{proof}[Proof of Corollary~\ref{cor_vol_lpf}]
Corollary~\ref{cor_vol_lpf} can be obtained by replacing the bound on the lattice diameter used in
the proof of the volume bound in \cite{MR2855866} by the sharp bound given in Theorem~\ref{lpf_ld_bound}.
We follow the lines of the proof in \cite{MR2855866}. 
It can be shown that $P$ is a subset of an integral polytope $Q$ whose number of interior integral points is between $1$ and $G(P) := |P \cap \Z^d|$. 
Since a $d$-dimensional   integral polytope with $k \in \N$ interior integral points is known to have volume of order $k 2^{2^{2d+o(d)}}$ (see \cite{MR1996360}),
we conclude $\vol(P) \le \vol(Q) \le G(P) 2^{2^{2 d + o(d)}}$. Thus, if we use the sharp bound on $\ld{P}$ and the implied bound $G(P) \le (\ld{P} + 1)^d \le 2^{d2^{d-1}}$ 
presented at page~\pageref{lpe_inequalities}, we arrive at the bound in Corollary~\ref{cor_vol_lpf}. 
\end{proof}

\begin{remark}\thmtitle{Rationality assumption in the definition of $\Plm$}
We remark that if $P+g$ is not lattice-free for all rational lines $g$ passing through the origin, then $P+g$ is also not lattice-free for all lines $g$ 
passing through the origin (not necessarily rational ones). This follows directly from the facts that every lattice-free set is contained in an 
inclusion-maximal lattice-free set and that the recession cone of every inclusion-maximal lattice-free set is a linear space spanned by rational vectors 
(for more information see \cite[\S 3]{MR1114315} or \cite[Theorem 1]{MR3027668}). 
Thus, the rationality assumption on $g$ in Theorem~\ref{lpf_ld_bound} is not relevant for the definition of $\Plm$. 
\end{remark}

\section{Proof of results from toric geometry (Theorems~\ref{ag1} and \ref{ag2})}\label{sec:ag}

\begin{proof}[Proof of Theorems~\ref{ag1}~and~\ref{ag2}.]
We will assume that $d \ge 4$, since these results are known to hold for $d\le3$; see \cite{MR2760660}. Let $X$ be a $d$-dimensional $\Q$-factorial toric Fano variety with Picard number one and at most canonical singularities. In this case, 
the associated fan is spanned by the faces of an integral simplex $S \subseteq \R^d$ with the origin as its only interior integral point. We note that, in general,
$S^*$ does not have to be an integral simplex.

By the toric dictionary (see \cite{toricbook}), 
\[(-K_X)^d = d!\, \vol_\Z(S^{\ast}) \le 2 (s_d-1)^2,\]
where the inequality follows from Theorem~\ref{thm_dual_face_volumes} (with $l=d$).
 Moreover, equality holds if and only if $S^* \cong S^d_1$. As explained in \cite{ben} (in particular, Theorem~3.3 and Proposition~4.4), this yields  
$X \cong \P\left(\frac{2 (s_d-1)}{s_1}, \ldots, \frac{2 (s_d-1)}{s_{d-1}}, 1, 1\right)$. This proves Theorem~\ref{ag1}.

Regarding Theorem~\ref{ag2}, it is well-known (see \cite{Lat96}) that the maximal anticanonical degree $(-K_X).C$ of a torus-invariant integral curve $C$ on $X$ equals 
\[\max_{F \in \cF_1(S^{\ast})} \vol_\Z(F).\]
Again, Theorem~\ref{thm_dual_face_volumes} (with $l=1$) yields that this value is at most $2 (s_d - 1)$. Here, the equality case is also uniquely determined by the same $X$ as 
in Theorem~\ref{ag1}(c).
\end{proof}


\begin{thebibliography}{CDPDS{\etalchar{+}}12}

\bibitem[Ave12]{MR2967480}
G.~Averkov, \emph{On the size of lattice simplices with a single interior
  lattice point}, SIAM J. Discrete Math. \textbf{26} (2012), no.~2, 515--526.

\bibitem[Ave13]{MR3027668}
\bysame, \emph{A proof of {L}ov\'asz's theorem on maximal lattice-free sets},
  Beitr. Algebra Geom. \textbf{54} (2013), no.~1, 105--109.

\bibitem[AWW11]{MR2855866}
G.~Averkov, C.~Wagner, and R.~Weismantel, \emph{Maximal lattice-free polyhedra:
  finiteness and an explicit description in dimension three}, Math. Oper. Res.
  \textbf{36} (2011), no.~4, 721--742.

\bibitem[Bal71]{Balas71}
E.~Balas, \emph{Intersection cuts --- a new type of cutting planes for integer
  programming}, Operations Res. \textbf{19} (1971), 19--39.

\bibitem[Bar02]{MR1940576}
A.~Barvinok, \emph{A {C}ourse in {C}onvexity}, Graduate Studies in Mathematics,
  vol.~54, American Mathematical Society, Providence, RI, 2002.

\bibitem[Bat94]{MR1269718}
V.~V. Batyrev, \emph{Dual polyhedra and mirror symmetry for {C}alabi-{Y}au
  hypersurfaces in toric varieties}, J. Algebraic Geom. \textbf{3} (1994),
  no.~3, 493--535.

\bibitem[BB92]{MR1166957}
A.~A. Borisov and L.~A. Borisov, \emph{Singular toric {F}ano three-folds}, Mat.
  Sb. \textbf{183} (1992), no.~2, 134--141.

\bibitem[Bli14]{MR1500976}
H.~F. Blichfeldt, \emph{A new principle in the geometry of numbers, with some
  applications}, Trans. Amer. Math. Soc. \textbf{15} (1914), no.~3, 227--235.

\bibitem[Buc08]{fake}
W.~Buczynska, \emph{Fake weighted projective spaces},
  http://arxiv.org/abs/0805.1211, 2008.

\bibitem[Cas97]{MR1434478}
J.~W.~S. Cassels, \emph{An {I}ntroduction to the {G}eometry of {N}umbers}, Classics
  in Mathematics, Springer-Verlag, Berlin, 1997, Corrected reprint of the 1971
  edition.

\bibitem[CDPDS{\etalchar{+}}12]{arxiv:math/1211.0388}
M.~Conforti, A.~Del~Pia, M.~Di~Summa, Y.~Faenza, and R.~Grappe, \emph{Reverse
  {C}hv{\'a}tal-{G}omory rank}, in M. Goemans, J. Correa (Eds.), Proceedings of 
  the XVI International Conference on Integer Programming and Combinatorial 
  Optimization (IPCO), Lecture Notes in Computer Science 7801, 
  Springer-Verlag (2013), pp. 133-144.

\bibitem[CLS11]{toricbook}
D.~A. Cox, J.~B. Little, and H.~K. Schenck, \emph{Toric {V}arieties}, Graduate
  Studies in Mathematics, vol. 124, American Mathematical Society, Providence,
  RI, 2011.

\bibitem[Cur22]{MR1520110}
D.~R. Curtiss, \emph{On {K}ellogg's {D}iophantine {P}roblem}, Amer. Math.
  Monthly \textbf{29} (1922), no.~10, 380--387.

\bibitem[DPW12]{MR2968262}
A.~Del~Pia and R.~Weismantel, \emph{On convergence in mixed integer
  programming}, Math. Program. \textbf{135} (2012), no.~1-2, Ser. A, 397--412.

\bibitem[Erd50]{MR0043117}
P.~Erd{\H o}s, \emph{On a {D}iophantine equation}, Mat. Lapok \textbf{1}
  (1950), 192--210.

\bibitem[GL87]{MR893813}
P.~M. Gruber and C.~G. Lekkerkerker, \emph{Geometry of {N}umbers}, second ed.,
  North-Holland Mathematical Library, vol.~37, North-Holland Publishing Co.,
  Amsterdam, 1987.

\bibitem[Hen83]{MR688412}
D.~Hensley, \emph{Lattice vertex polytopes with interior lattice points},
  Pacific J. Math. \textbf{105} (1983), no.~1, 183--191.

\bibitem[HLP52]{MR0046395}
G.~H. Hardy, J.~E. Littlewood, and G.~P{\'o}lya, \emph{Inequalities},
  Cambridge, at the University Press, 1952, 2d ed.

\bibitem[IK87]{IK_original}
O.~Izhboldin and L.~Kurliandchik, \emph{Partition of unity}, Kvant, no.~7,
  1987, pp.~48--52 (in Russian).

\bibitem[IK95]{MR1363298}
\bysame, \emph{Unit fractions}, Proceedings of the {S}t.\ {P}etersburg
  {M}athematical {S}ociety, {V}ol.\ {III} (Providence, RI), Amer. Math. Soc.
  Transl. Ser. 2, vol. 166, Amer. Math. Soc., 1995, pp.~193--200.

\bibitem[Kas09]{MR2549542}
A.~M. Kasprzyk, \emph{Bounds on fake weighted projective space}, Kodai Math. J.
  \textbf{32} (2009), no.~2, 197--208.

\bibitem[Kas10]{MR2760660}
\bysame, \emph{Canonical toric {F}ano threefolds}, Canad. J. Math. \textbf{62}
  (2010), no.~6, 1293--1309.

\bibitem[Kas13]{fake2}
\bysame, \emph{Classifying terminal weighted projective space},
  http://arXiv:1304.3029v1, 2013.

\bibitem[KL88]{MR970611}
R.~Kannan and L.~Lov{\'a}sz, \emph{Covering minima and lattice-point-free
  convex bodies}, Ann. of Math. (2) \textbf{128} (1988), no.~3, 577--602.

\bibitem[Lat96]{Lat96}
R.~Laterveer, \emph{{Linear systems on toric varieties}}, Tohoku Math. J.
  \textbf{48} (1996), 451--458.

\bibitem[Law91]{MR1116368}
J.~Lawrence, \emph{Finite unions of closed subgroups of the {$n$}-dimensional
  torus}, Applied geometry and discrete mathematics, DIMACS Ser. Discrete Math.
  Theoret. Comput. Sci., vol.~4, Amer. Math. Soc., Providence, RI, 1991,
  pp.~433--441.

\bibitem[Lov89]{MR1114315}
L.~Lov{\'a}sz, \emph{Geometry of numbers and integer programming}, Mathematical
  programming ({T}okyo, 1988), Math. Appl. (Japanese Ser.), vol.~6, SCIPRESS,
  Tokyo, 1989, pp.~177--201.

\bibitem[LZ91]{MR1138580}
J.~C. Lagarias and G.~M. Ziegler, \emph{Bounds for lattice polytopes containing
  a fixed number of interior points in a sublattice}, Canad. J. Math.
  \textbf{43} (1991), no.~5, 1022--1035.

\bibitem[Mah39]{MR0001242}
K.~Mahler, \emph{Ein \"{U}bertragungsprinzip f\"ur konvexe {K}\"orper}, \v
  Casopis P\v est. Mat. Fys. \textbf{68} (1939), 93--102.

\bibitem[Nil05]{gorst}
B.~Nill, \emph{Gorenstein toric {F}ano varieties}, Manuscripta Math.
  \textbf{116} (2005), no.~2, 183--210.

\bibitem[Nil07]{ben}
\bysame, \emph{Volume and lattice points of reflexive simplices}, Discrete
  Comput. Geom. \textbf{37} (2007), no.~2, 301--320.

\bibitem[NZ11]{MR2832401}
B.~Nill and G.~M. Ziegler, \emph{Projecting lattice polytopes without interior
  lattice points}, Math. Oper. Res. \textbf{36} (2011), no.~3, 462--467.

\bibitem[Pik01]{MR1996360}
O.~Pikhurko, \emph{Lattice points in lattice polytopes}, Mathematika
  \textbf{48} (2001), no.~1-2, 15--24 (2003).

\bibitem[Rab89a]{MR1039134}
S.~Rabinowitz, \emph{A census of convex lattice polygons with at most one
  interior lattice point}, Ars Combin. \textbf{28} (1989), 83--96.

\bibitem[Rab89b]{MR1030772}
\bysame, \emph{A theorem about collinear lattice points}, Utilitas Math.
  \textbf{36} (1989), 93--95.


\bibitem[Roc97]{MR1451876}
R.~T. Rockafellar, \emph{Convex {A}nalysis}, Princeton Landmarks in Mathematics,
  Princeton University Press, Princeton, NJ, 1997, Reprint of the 1970
  original, Princeton Paperbacks.

\bibitem[Saw54]{MR0061139}
D.~B. Sawyer, \emph{The lattice determinants of asymmetrical convex regions},
  J. London Math. Soc. \textbf{29} (1954), 251--254.

\bibitem[Sch86]{MR874114}
A.~Schrijver, \emph{Theory of {L}inear and {I}nteger {P}rogramming},
  Wiley-Interscience Series in Discrete Mathematics, John Wiley \& Sons Ltd.,
  Chichester, 1986, A Wiley-Interscience Publication.

\bibitem[Sch93]{MR1216521}
R.~Schneider, \emph{Convex {B}odies: {T}he {B}runn-{M}inkowski {T}heory},
  Encyclopedia of Mathematics and its Applications, vol.~44, Cambridge
  University Press, Cambridge, 1993.

\bibitem[Sou05]{arxiv:math/0502247v1}
K.~Soundararajan, \emph{Approximating 1 from below using $n$ egyptian
  fractions}, Preprint arXiv:math/0502247v1, 2005.

\bibitem[ZPW82]{MR651251}
J.~Zaks, M.~A. Perles, and J.~M. Wills, \emph{On lattice polytopes having
  interior lattice points}, Elem. Math. \textbf{37} (1982), no.~2, 44--46.

\end{thebibliography}

\newcommand{\etalchar}[1]{$^{#1}$}
\providecommand{\bysame}{\leavevmode\hbox to3em{\hrulefill}\thinspace}
\providecommand{\MR}{\relax\ifhmode\unskip\space\fi MR }
\providecommand{\MRhref}[2]{%
  \href{http://www.ams.org/mathscinet-getitem?mr=#1}{#2}
}
\providecommand{\href}[2]{#2}

\end{document}